\documentclass[12pt]{amsart}
\usepackage{amssymb,latexsym}
\usepackage[mathscr]{eucal}
\usepackage{graphics}

\theoremstyle{plain}
\newtheorem{theorem}{Theorem}[section]
\newtheorem{corollary}[theorem]{Corollary}
\newtheorem{lemma}[theorem]{Lemma}
\newtheorem{proposition}[theorem]{Proposition}
\newtheorem{remark}[theorem]{Remark}
\newtheorem{deff}[theorem]{Definition}
\newtheorem{example}[theorem]{Example}

\DeclareMathOperator{\diag}{diag}
\DeclareMathOperator{\und}{und}
\DeclareMathOperator{\lnd}{lnd}
\DeclareMathOperator{\uni}{uni}
\DeclareMathOperator{\lni}{lni}
\DeclareMathOperator{\se}{se}
\DeclareMathOperator{\scop}{sc}
\DeclareMathOperator{\cl}{cl}

\newcommand{\xo}{\frac{\xi}{\omega}}
\renewcommand{\mathcal}{\mathscr}
\begin{document}

\title[Soft Ideals and Arithmetic Mean Ideals]{Soft Ideals and Arithmetic Mean Ideals}

\author[Victor Kaftal]{Victor Kaftal}

\address{University of Cincinnati\\
Department of mathematical sciences\\ Cincinnati, OH 45221-0025\\
USA}

\email{victor.kaftal@uc.edu}

\author[Gary Weiss]{Gary Weiss}

\address{University of Cincinnati\\
Department of mathematical sciences\\ Cincinnati, OH 45221-0025\\
USA}

\email{gary.weiss@uc.edu}

\thanks{Both authors were partially supported by grants of the Charles
Phelps Taft Research Center; the second named author was partially
supported by NSF Grants DMS 95-03062 and DMS 97-06911.}



\begin{abstract}
This article investigates the soft-interior (se) and the soft-cover
(sc) of operator ideals. These operations, and especially the first
one, have been widely used before, but making their role explicit
and analyzing their interplay with the arithmetic mean operations is
essential for the study  in \cite{10} of the multiplicity of traces. 
Many classical ideals are ``soft'',
i.e., coincide with their soft interior or with their soft cover,
and many ideal constructions yield soft ideals. Arithmetic mean (am)
operations were proven to be intrinsic to the theory of operator
ideals in \cite{6, 7} and arithmetic mean operations at infinity
(am-$\infty$) were studied in \cite{10}. Here we focus on the
commutation relations between these operations and soft operations.
In the process we characterize the am-interior and the am-$\infty$
interior of an ideal.
\end{abstract}

\maketitle

\section{Introduction}\label{sec:1}
Central to the theory of operator ideals (the two-sided ideals of the
algebra $B(H)$ of bounded operators on a separable Hilbert space
$H$)  are the notions of the commutator space of an ideal $I$ (the linear span of the
commutators $TA - AT$, $A \in I$, $T \in B(H)$) and of a trace
supported by the ideal. In this context,
the arithmetic (Cesaro) mean of monotone sequences first appeared
implicitly in \cite{20}, then played  in \cite{14} an explicit and  key role for determining the
 commutator space of the trace class,
and more recently entered center stage in \cite{6, 7} by providing the framework
for the characterization of the commutator space of arbitrary ideals. 
This prompted \cite{7} to associate more formally to a
given ideal $I$ the arithmetic mean ideals $I_a$, $_aI$, $I^o =
(_aI)_a$ (the am-interior of $I$) and $I^- = \text{$_a$}(I_a)$ (the
am-closure of $I$). (See Section 2 for definitions.) In particular, the
arithmetically mean closed ideals (those equal to their am-closure)
played an important role in the study of single commutators in
\cite{7}.

This paper and \cite{10}-\cite{12A} are part of an ongoing program
announced in \cite{9} dedicated to the study of arithmetic mean ideals
and their applications.

In \cite{10} we investigated the question: ``How many traces (up to
scalar multiples) can an ideal support?'' We found that for the
following two classes of ideals which we call ``soft'' the answer
is always zero, one or uncountably many: the soft-edged ideals
that coincide with their soft-interior $\se I: = IK(H)$ and the
soft-complemented ideals that coincide with their soft complement
$\scop I: = I : K(H)$ ($K(H)$ is the ideal of compact operators on
$H$ and for quotients of ideals see Section~\ref{sec:3}).

Softness properties have often played a role in the theory of
operator ideals, albeit that role was mainly implicit and sometimes
hidden. Taking the product of $I$ by $K(H)$ corresponds at the
sequence level to the ``little $o$'' operation, which figures so
frequently in operator ideal techniques. M. Wodzicki employs
explicitly the notion of soft interior of an ideal (although he does
not use this terminology) to investigate obstructions to the
existence of positive traces on an ideal (see 
\cite[Lemma 2.15,
Corollary 2.17]{21}. A special but important case of quotient is the
celebrated K\"{o}the dual of an ideal and general quotients have
been studied among others by Salinas \cite{17}. But to the best of
our knowledge the power of combining these two soft operations has
gone unnoticed along with their investigation and a systematic use
of their properties. Doing just that permitted us in \cite{10} to
considerably extend and simplify our study of the codimension of
commutator spaces.
In particular, we depended in a crucial way on
the interplay between soft operations and arithmetic mean
operations.

Arithmetic mean operations on ideals were first introduced in
\cite{7} and further studied in \cite{10}. For summable sequences,
the arithmetic mean must be replaced by the arithmetic mean at
infinity (am-$\infty$ for short), see for instance \cite{1, 7, 13,
21}. In \cite{10} we defined am-$\infty$ ideals and found that their
theory is in a sense dual to the theory of am-ideals, including the
role of $\infty$-regular sequences studied in
\cite[Theorem 4.12]{10}). In \cite{10} we considered only the ideals $_aI$, $I_a$,
$_{a_\infty}I$, and $I_{a_\infty}$, and so in this paper we focus
mostly on the other am and am-$\infty$ ideals.

In Section~\ref{sec:2} we prove that the sum of two am-closed ideals is am-closed 
(Theorem~\ref{thm:2.5}) by using the connection between 
majorization of infinite sequences and infinite substochastic matrices due to Markus \cite{15}. 
(Recent outgrowths from [ibid] from the classical theory for finite sequences and stochastic matrices to the 
infinite is one focus of \cite{11}.)
This leads naturally to defining a largest 
am-closed ideal $I_{-} \subset I$. 
We prove that $I_{-} = {}_aI$ for countably generated ideals 
(Theorem~\ref{thm:2.9}) while in general the
inclusion is proper. An 
immediate consequence is that a countably
generated ideal is am-closed $(I = I^-)$ 
if and only if it is
am-stable $(I = I_a)$ (Theorem~\ref{thm:2.11}). This generalizes a
result from \cite[Theorem 3.11]{2}. Then we prove that arbitrary
intersections of am-open ideals must be am-open
(Theorem~\ref{thm:2.17}) by first obtaining a characterization of the
am-interior of a principal ideal (Lemma~\ref{lem:2.14}) and then of
an arbitrary ideal (Corollary~\ref{cor:2.16}). This leads naturally
to defining the smallest am-open ideal $I^{oo} \supset I$.

In Section~\ref{sec:3} we obtain analogous results for the
am-$\infty$ case. But while the statement are similar, the
techniques employed in proving them are often substantially
different. For instance, the proof that the sum of two am-$\infty$
closed ideals is am-$\infty$ closed (Theorem~\ref{thm:3.2}) depends
on a $w^*$-compactness argument rather than a matricial one.

In Section~\ref{sec:4} we study soft ideals. The soft-interior $\se
I$ and the soft-cover $\scop I$ are, respectively, the largest
soft-edged ideal contained in $I$ and the smallest soft-complemented ideal
 containing $I$. The pair $\se I \subset \scop I$ is the
generic example of what we call a soft pair. Many classical ideals,
i.e., ideals whose characteristic set is a classical sequence space,
turn out to be soft. Among soft-edged ideals are minimal Banach
ideals $\mathfrak{S}^{(o)}_\phi$ for a symmetric norming function
$\phi$, Lorentz ideals $\mathcal{L}(\phi)$, small Orlicz ideals
$\mathcal{L}_M^{(o)}$, and idempotent ideals.

To prove soft-complementedness of an ideal we often find it
convenient to prove instead a stronger property which we call strong
soft-complementedness (Definition~\ref{def:4.4},
Proposition~\ref{prop:4.5}). Among strongly soft complemented ideals
are principal and more generally countably generated ideals, maximal
Banach ideals ideals $\mathfrak{S}_\phi$, Lorentz ideals
$\mathcal{L}(\phi)$, Marcinkiewicz ideals $_a(\xi)$, and Orlicz
ideals $\mathcal{L}_M^{(o)}$. K\"{o}the duals and idempotent ideals
are always soft-complemented but can fail to be strongly
soft-complemented.

Employing the properties of soft pairs for the embedding
$\mathfrak{S}_\phi^{(o)} \subset  \mathfrak{S}_\phi$ in the
principal ideal case, we present a simple proof of the fact that if
a principal ideal is a Banach ideal then its generator must be
regular, which is due to Allen and Shen \cite[Theorem 3.23]{2} and
was also obtained by Varga \cite{19} (see Remark~\ref{rem:4.8}(iv)
and \cite[Theorem 5.20]{7}). The same property of the embedding yields
a simpler proof of part of a result by Salinas in  \cite[Theorem 2.3]{17}. 
Several results relating small Orlicz and Orlicz ideals given in theorems in \cite{7}
follow immediately from the fact that $\mathcal{L}_M^{(o)} \subset
\mathcal{L}_M$ are also soft pairs (see remarks after
Proposition~\ref{prop:4.11}).

Various operations on ideals produce additional soft ideals. 
Powers of
soft-edged ideals, directed unions (by inclusion) of soft-edged ideals, 
finite intersections and finite sums of soft-edged ideals
are all soft-edged. Powers of soft-complemented ideals and arbitrary
intersections of soft-complemented ideals are also soft-comp\-lemented
(Section~\ref{sec:4}). As consequences follow the softness properties
of the am and am-$\infty$ stabilizers of the
trace-class $\mathcal{L}_1$ (see Sections \ref{sec:2} (\P 4) and \ref{sec:3}(\P 2) for the 
definitions)
which play an important role in \cite{9}-\cite{10}.
However, whether the sum of two soft-complemented ideals or even two
strongly soft-complemented ideals
is always soft-complemented
remains unknown. We prove that it is under the
additional hypothesis
that one of the ideals is countably generated and the other is
strongly soft-complemented (Theorem~\ref{thm:5.7}).

Some of the commutation relations between the soft-interior and
soft-cover operations and the am and am-$\infty$ operations played a
key role in \cite{10}. 
We investigate the commutation relations with
the remaining operations in 
Section~\ref{sec:6}
(Theorems~\ref{thm:6.1}, \ref{thm:6.4}, \ref{thm:6.9}, and
\ref{thm:6.10}). As a consequence we obtain which operations
preserve soft-complementedness and soft-edgedness. Some of the
relations remain open, e.g., we do not know if $\scop I_a = (\scop
I)_a$ (see Proposition~\ref{prop:6.8}).

Following this paper in the program outlined in \cite{9} will be
\cite{11} where we clarify the interplay between arithmetic mean
operations, infinite convexity, and diagonal invariance and
\cite{12} where we investigate the lattice properties of several
classes of operator ideals proving results of the kind: between two
proper ideals, at least one of which is am-stable (resp.,
am-$\infty$  stable) lies a third am-stable (resp., am-$\infty$
stable) principal ideal and applying them to various arithmetic
mean cancellation and inclusion properties (see \cite[Theorem 11 and
Propositions 12--14]{9}. Example, for which ideals $I$ does the $I_a
= J_a$ (resp., $I_a \subset J_a$, $I_a \supset J_a$ ) imply $I = J$
(resp., $I \subset J$, $I \supset J$) and in the latter cases, is
there an ``optimal'' $J$?

\section{Preliminaries and Arithmetic Mean Ideals}\label{sec:2}

Calkin \cite{5} established a correspondence between
\textit{two-sided ideals} of bounded operators on a complex
separable infinite dimensional Hilbert space and
\textit{characteristic sets}, i.e., hereditary (i.e., solid) cones
$\Sigma \subset c_o^*$ (the collection of sequences decreasing to
0), that are invariant under ampliations. For each $m \in
\mathbb{N}$ , the $m$-fold ampliation $D_m$ is defined by:
\[
c_o^* \backepsilon \xi\longrightarrow  D_m \xi := \langle
\xi_1,\ldots, \xi_1, \xi_2,\ldots, \xi_2, \xi_3,\ldots,
\xi_3,\ldots\rangle
\]
with each entry $\xi_i$ repeated $m$ times. The Calkin
correspondence $I \rightarrow \Sigma(I)$ induced by $I \owns
X \rightarrow s(X) \in \Sigma(I)$, where $s(X)$ denotes the sequence
of the $s$-numbers of $X$, is a lattice isomorphism between ideals
and characteristic sets and its inverse is the map from a
characteristic set $\Sigma$ to the ideal generated by the collection
of the diagonal operators $\{\diag \xi \mid \xi \in \Sigma\}$. For a
sequence $0\leq  \xi \in c_o$, denote by $\xi^* \in c_o^*$ the
decreasing rearrangement of $\xi$, and for each $\xi \in c_o^*$ denote by
$(\xi)$ the principal ideal generated by $\diag \xi$, so that
$(s(X))$ denotes the principal ideal generated by the operator 
$X \in K(H)$ (the ideal of compact operators on the Hilbert space $H$).\linebreak
Recall that $\eta \in \Sigma((\xi))$ precisely when  $\eta=
O(D_m\xi)$ for some $m$. Thus the equivalence between $\xi$ and
$\eta$ ($\xi \asymp\eta$ if $\xi = O(\eta)$ and  $\eta= O(\xi)$) is
only sufficient for $(\xi) = (\eta)$. It is also necessary if one of
the two sequences (and hence both) satisfy the
$\Delta_{1/2}$-condition. Following the notations of \cite{21}, we
say that $\xi$ satisfies the $\Delta_{1/2}$-condition if
$\sup\frac{\xi_n}{\xi_{2n}} < \infty$, i.e., $D_2\xi = O(\xi)$,
which holds if and only if $D_m\xi = O(\xi)$ for all $m \in
\mathbb{N}$.

Dykema, Figiel, Weiss and Wodzicki \cite{6, 7} showed that the (Cesaro)
\textit{arithmetic mean} plays an essential role in the theory of operator ideals by 
using it to characterize the normal operators in the commutator space of an ideal. 
(The commutator space $[I, B(H)]$ of an ideal $I$, also called the commutator ideal of
$I$, is the span of the commutators of elements of $I$ with elements of $B(H)$).  
This led them to introduce and study the arithmetic mean
and pre-arithmetic mean of an ideal and the consequent notions of
am-interior and am-closure of an ideal.

The arithmetic mean of any sequence $\eta$ is the sequence 
$\eta_a\,:=\, \langle\,\frac1n \sum_{i=1}^n\,\eta_i \rangle$. 
For every ideal $I$, the pre-arithmetic mean ideal $_aI$
and the arithmetic mean ideal $I_a$ are the ideals with
characteristic sets
\begin{align*}
\Sigma(_aI) &= \{\xi \in c_o^* \mid \xi_a \in \Sigma(I)\}
\\
\Sigma(I_a) &= \{\xi \in c_o^* \mid \xi=O(\eta) \text{ for some
}\eta\in \Sigma(I)\}.
\end{align*}

A consequence of one of the main results in \cite[Theorem 5.6]{7} is
that the positive part of the commutator space $[I, B(H)]$ coincides
with the positive part of the pre-arithmetic mean ideal $_aI$, that is: 
\[
[I,B(H)]^+= (_aI)^+
\]
In particular, ideals that fail to support any nonzero trace, i.e., ideals for
which $I = [I, B(H)]$, are precisely those for which $I =
\text{$_a$}I$ (or, equivalently, $I = I_a$) and are called
arithmetically mean stable (am-stable for short). 
The smallest nonzero am-stable ideal is the upper stabilizer of the trace-class ideal $\mathcal{L}_1$ (in the notations of \cite{7})
\[
st^a(\mathcal{L}_1) :=
\bigcup^\infty_{m=0}(\omega)_{a^m}=
\bigcup^\infty_{m=0}(\omega\log^m)
\]
 where $\omega = \langle
1/n\rangle$ denotes the harmonic sequence (see
\cite[Proposition~4.18]{10}). There is no largest proper am-stable
ideal. Am-stability for many classical ideals was studied in
\cite[Sections 5.11--5.27]{7}.

Arithmetic mean operations on ideals were introduced in \cite[Sections
2.8 and 4.3]{7} and employed, in particular, in the study of single
commutators \cite[Section 7]{7}: the arithmetic mean closure $I^-$
and the arithmetic mean interior $I^o$ of an ideal $I$ are defined
respectively as $I^- := \text{$_a$}(I_a)$ and $I^o := (_aI)_a$. The
following 5-chain inclusion holds:
\[
_aI \subset  I^o \subset I
\subset  I^- \subset  I_a
\]
Ideals that coincide with their
am-closure (resp., am-interior) are called am-closed (resp., 
am-open), and $I^-$ is the smallest am-closed ideal containing $I$
(resp.,
$I^o$ is the largest am-open ideal contained in $I$). We
list here some of the elementary properties of am-closed and am-open
ideals, and since there is a certain symmetry between them, we shall
consider both in parallel.

An ideal $I$ is am-closed (resp., am-open) if and only if $I =
\text{$_a$}J$ (resp., $I = J_a$) for some ideal $J$. The necessity
follows from the definition of $I^-$ (resp., $I^o$) and the
sufficiency follows from the identities $I_a = (_a(I_a))_a$ and $_aI
= \text{$_a$}((_aI)_a)$ that are simple consequences of the 5-chain
of inclusions listed above.

The characteristic set  $\Sigma(\mathcal{L}_1)$ of the trace-class ideal is $ \ell_1^*$, 
the collection of monotone nonincreasing nonnegative summable sequences.
It is elementary to show $\mathcal{L}_1 = \text{$_a$}(\omega)$,
 $\mathcal{L}_1$ is the smallest nonzero am-closed ideal,  $(\omega) = F_a = (\mathcal{L}_1)_a$, 
 and so $(\omega)$ is the smallest nonzero am-open ideal ($F$ denotes the finite rank
ideal.)

In terms of characteristic sets:
\[
\Sigma(I^-) = \{\xi \in  c_o^* \mid \xi_a \leq \eta_a \text{ for
some }\eta\in \Sigma(I)\}
\]
\[
\Sigma(I^o) = \{\xi \in  c_o^* \mid \xi \leq \eta_a \in \Sigma(I)
\text{ for some }\eta\in c_o^*\}
\]
Here and throughout, the relation between sequences ``$\leq$''
denotes pointwise, i.e., for all $n$. The relation $\xi\prec\eta$ defined by $\xi_a
\leq \eta_a$ is called majorization and plays an important role in
convexity theory (e.g., see \cite{15,16}). We will investigate it
further in this context in \cite{11} (see also \cite{9}). But for
now, notice that $I$ is am-closed if and only if $\Sigma(I)$ is
hereditary (i.e., solid) under majorization.

The two main results in this section are that the (finite) sum of
am-closed ideals is am-closed and that intersections of am-open
ideals are am-open. These will lead to two additional natural
arithmetic mean ideal operations, $I_-$ and $I^{oo}$, see
Corollary~\ref{cor:2.6} and Definition~\ref{def:2.18}.

We start by determining how the arithmetic mean operations
distribute with respect to direct unions and intersections of ideals
and with respect to finite sums. Recall that the union of a
collection of ideals that is directed by inclusion and the
intersection of an arbitrary collection of ideals are ideals. The
proofs of the following three lemmas are elementary, with the
exception of one of the inclusions in Lemma~\ref{lem:2.2}(iii)
which is a simple consequence of Theorem~\ref{thm:2.17} below.

\begin{lemma}\label{lem:2.1}
For $\{I_\gamma, \gamma\in\Gamma\}$ a collection of ideals directed by inclusion:

\item[(i)]
$_a(\bigcup_\gamma I_\gamma)= \bigcup_\gamma{} _a(I_\gamma)$

\item[(ii)]
$(\bigcup_\gamma I_\gamma)_a= \bigcup_\gamma (I_\gamma)_a$

\item[(iii)]
$(\bigcup_\gamma I_\gamma)^o= \bigcup_\gamma (I_\gamma)^o$

\item[(iv)]
$(\bigcup_\gamma I_\gamma)^-= \bigcup_\gamma (I_\gamma)^-$

\item[(v)]
If all $I_\gamma$ are am-stable, (resp., am-open, am-closed) then
$\bigcup_\gamma I_\gamma$ is am-stable, (resp., am-open, am-closed).
\end{lemma}

\begin{lemma} \label{lem:2.2}
For $\{I_\gamma, \gamma\in\Gamma\}$ a collection of ideals:

\begin{enumerate}
\item[(i)]
$_a(\bigcap_\gamma I_\gamma)= \bigcap_\gamma{} _a(I_\gamma)$

\item[(ii)]
$(\bigcap_\gamma I_\gamma)_a \subset \bigcap_\gamma (I_\gamma)_a$
\quad
(inclusion can be proper by Example~\ref{ex:2.4}(i))

\item[(iii)]
$(\bigcap_\gamma I_\gamma)^o= \bigcap_\gamma (I_\gamma)^o$
\quad
(equality holds by Theorem~\ref{thm:2.17})

\item[(iv)]
$(\bigcap_\gamma I_\gamma)^- \subset \bigcap_\gamma (I_\gamma)^-$
\quad (inclusion can be proper by Example~\ref{ex:2.4}(i))

\item[(v)]
If all $I_\gamma$ are am-stable, (resp., am-open, am-closed) then
$\bigcap_\gamma I_\gamma$ is am-stable, (resp., am-open, am-closed).
\end{enumerate}
\end{lemma}

\begin{lemma} \label{lem:2.3}
For all ideals $I$, $J$:

\item[(i)]
$I_a + J_a = (I + J)_a$

\item[(ii)]
$_aI + \text{$_a$}J \subset  \text{$_a$}(I + J)$ \quad (the
inclusion can be proper by Example~\ref{ex:2.4}(ii))

\item[(iii)]
$I^o + J^o \subset (I + J)^o$
\quad
(the inclusion can be proper by Example~\ref{ex:2.4}(ii))

\item[(iv)]
$I^- + J^- \subset (I + J)^-$ \quad (equality is Theorem~\ref{thm:2.5})

\item[(v)]
If $I$ and $J$ are am-open, so is $I + J$.
\end{lemma}

\begin{example}\label{ex:2.4}
(i) In general, equality does not hold in Lemma~\ref{lem:2.2}(ii)
or, equivalently, in (iv) even when $\Gamma$ is finite. Indeed it is
easy to construct two nonsummable sequences $\xi$  and $\eta$  in
$c_o^*$ such that $\min(\xi , \eta)$ is summable. But then, as it is 
elementary to show, $(\xi ) \cap (\eta) = (\min(\xi , \eta))$ and
hence $((\xi ) \cap (\eta))_a = (\omega)$ while 
\[
(\xi )_a \cap
(\eta)_a = (\xi_a) \cap (\eta_ a) = (\min(\xi_ a, \eta_a))
\supsetneqq (\omega), 
\]
the inclusion since  $\omega = o(\xi_ a)$, $\omega = o(\eta_a)$, 
hence $\omega = o(\min(\xi_ a, \eta_a))$, and the inequality since $\omega$ satisfies the 
$\Delta_{1/2}$-condition and then equality leads to a contradiction.

(ii) In general, equality does not hold in Lemma~\ref{lem:2.3}(ii)
or (iii). Indeed take the principal ideals generated by two
sequences $\xi$  and $\eta$  in $c_o^*$ such that $\xi  + \eta =
\omega$ but $\omega \neq O(\xi )$ and $\omega \neq O(\eta)$, which
implies that
\[
_a(\xi ) = \text{$_a$}(\eta) = \{0\} \neq \mathcal{L}_1 =
\text{$_a$}(\omega) = \text{$_a$}((\xi ) + (\eta)).
\]

The same example shows that
\[
(\xi)^o + (\eta)^o = \{0\} \neq (\omega) = ((\xi ) + (\eta))^o.
\]
\end{example}

That the sum of finitely many am-open ideals is am-open (Lemma~\ref{lem:2.3}(v)), is an immediate consequence of
Lemma~\ref{lem:2.3}(iii). Less trivial is the fact that the sum of
finitely many am-closed ideals is am-closed, or, equivalently, that
equality holds in Lemma~\ref{lem:2.3}(iv). This result was
announced in \cite{9}. The proof we present here exploits the role
of substochastic matrices in majorization theory (\cite{15}, see
also \cite{11}). Recall that a matrix $P$ is called substochastic if
$P_{ij} \geq 0$, $\sum_{i=1}^\infty P_{ij}\leq 1$  for all $j$ and
$\sum_{j=1}^\infty P_{ij}\leq 1$ for all $i$. By extending the
well-known result for finite sequence majorization (e.g., see
\cite{16}), Markus showed in \cite[Lemma 3.1]{15} that if $\eta$,
$\xi  \in c_o^*$, then $\eta_a \leq \xi_ a$ if and only if there is
a substochastic matrix $P$ such that $ \eta = P\xi$. Finally, recall
also the Calkin \cite{5} isomorphism between proper two sided ideals
of $B(H)$ and ideals of $\ell_\infty$  that associates to an ideal
$J$ the symmetric sequence space $S(J)$ defined by $S(J) := \{\eta
\in c_o \mid \diag \eta \in J\}$ (e.g., see \cite{5} or
\cite[Introduction]{7}). It is immediate to see that $S(J) = \{\eta
\in c_o \mid |\eta|^* \in \Sigma(J)\}$ and that for any two ideals,
$S(I + J) = S(I) + S(J)$.

\begin{theorem} \label{thm:2.5}
$(I + J)^- = I^- + J^-$ for all ideals $I$, $J$. 

In particular, the sum of two am-closed ideals is am-closed.
\end{theorem}

\begin{proof}
The inclusion $I^- + J^- \subset  (I + J)^-$ is elementary and was stated in
Lemma~\ref{lem:2.3}(iv). 
Let $\xi  \in \Sigma((I+ J)^-)$,
then $\xi_ a \in \Sigma((I + J)_a)$ so that $\xi_a \leq (\rho + \eta)_a$ for some $\rho\in \Sigma(I)$ and $\eta \in \Sigma(J)$. 
Then by Markus' lemma \cite[Lemma 3.1]{15}, there is a substochastic matrix $P$ such that $\xi  = P(\rho + \eta)$. 
Let $\Pi$ be a permutation matrix monotonizing $P\rho$, i.e., $(P\rho)^* = \Pi P\rho$, then $\Pi P$ too is substochastic and hence by the same result, 
$((P\rho )^*)_a \leq \rho_a$, i.e., $(P\rho )^* \in
\Sigma(I^-)$, or equivalently, $P\rho  \in S(I^-)$. Likewise, $P\eta
\in S(J^-)$, whence $\xi  \in S(I^-) + S(J^-) = S(I^- + J^-)$ and
hence  $\xi  \in  \Sigma (I^- + J^-)$. Thus $(I + J)^- \subset I^- + J^-$, concluding the proof.
\end{proof}

As a consequence, the collection of all the am-closed ideals
contained in an ideal $I$ is directed and hence its union is an
am-closed ideal by Lemma~\ref{lem:2.1}(v).

\begin{corollary} \label{cor:2.6}
Every ideal $I$ contains a largest am-closed ideal denoted by $I_-$,
which is given by
\[
I_- := \bigcup\{J \mid J \subset I\text{ and }J \text{ is
am-closed}\}.
\]
\end{corollary}

Thus $I_- \subset I \subset  I^-$ and $I$ is am-closed if and only
if one of the inclusions and hence both of them are equalities.
Since $_aI \subset I$ and $_aI$ is am closed, $_aI \subset  I_-$.
The inclusion can be proper as seen by considering any am-closed but not
am-stable ideal $I$, e.g, $I = \mathcal{L}_1$ where
$_a(\mathcal{L}_1) = \{0\}$. If equality holds, we have the
following equivalences:

\begin{lemma} \label{lem:2.7}
For every ideal $I$, the following conditions are equivalent.

\item[(i)]
$I_- = \text{$_a$}I$

\item[(ii)]
If $J^- \subset I$ for some ideal $J$, then $J^- \subset ~_aI$.

\item[(iii)]
If $J^- \subset I$ for some ideal $J$, then $J_a \subset  I$.

\item[(iv)]
If $_aJ \subset I$ for some ideal $J$, then $J^o\subset  I$.
\end{lemma}
\noindent We leave the proof to the reader. Notice that the converses (ii)--(iv) hold trivially for any pair of ideals $I$ and $J$.

Theorem~\ref{thm:2.9} below will show that for countably generated
ideals the equality $_aI = I_-$ always holds, i.e., $_aI$ is the
largest am-closed ideal contained in $I$.

We first need the following lemma.

\begin{lemma} \label{lem:2.8}
If $I$ is a countably generated ideal and $\mathcal{L}_1\subset  I$,
then $(\omega)\subset I$.
In particular, $(\omega)$ is the smallest
principal ideal containing $\mathcal{L}_1$.
\end{lemma}
\begin{proof}
Let $\rho^{(k)}$ be a sequence of generators for the characteristic
set $\Sigma(I)$, i.e., for every $\xi  \in  \Sigma(I)$ there are $m,
k \in \mathbb{N}$  for which $\xi  = O(D_m \rho^{(k)})$. By adding
if necessary to this sequence of generators all their ampliations
and then by passing to the sequence $\rho^{(1)} + \rho^{(2)} +
\cdots + \rho^{(k)}$, we can assume that $\rho^{(k)} \leq
\rho^{(k+1)}$ and that then $\xi  \in  \Sigma(I)$ if and only if
$\xi  = O(\rho^{(m)})$ for some $m \in \mathbb{N}$. Thus if $\omega
\notin \Sigma(I)$ there is an increasing sequence of indices $n_k$
such that $(\frac{\omega}{\rho^{(k)}})_{n_k}\geq k^3$ for all $k
\geq 1$. Set $n_o := 0$ and define $\xi_ j := \frac{1}{k^2n_k}$ for
$n_{k-1}<j\leq  n_k$ and $k \geq 1$. Then it is immediate 
that $\xi  \in  \ell_1^*$. On the other hand, $\xi  \neq O(\rho^{(m)})$ for any $m \in \mathbb{N}$ since for every $k \geq m$,
\[
\bigg(\frac{\xi}{\rho^{(m)}} \bigg)_{n_k} \geq
\bigg(\frac{\xi}{\rho^{(k)}} \bigg)_{n_k}=
\frac{1}{k^2n_k\rho^{(k)}_{n_k}}\geq k.
\]
and hence $\xi  \notin \Sigma(I)$, against the hypothesis $\mathcal{L}_1 \subset I$.
\end{proof}

\begin{theorem}\label{thm:2.9}
If $I$ is a countably generated ideal, then $I_- = ~_aI$.
\end{theorem}

\begin{proof}
Let $\eta \in  \Sigma(I_-)$. Then $(\eta)^- \subset  I_- \subset I$.
We claim that $\eta_a \in  \Sigma(I)$, i.e., $I_- \subset
\text{$_a$}I$ and hence equality holds from the maximality of $I_-$.
If $0 \neq \eta \in  \ell_1^*$, then $(\eta)^- = \mathcal{L}_1$,
hence by Lemma~\ref{lem:2.8}, $(\omega)\subset I$ and thus $\eta_a
\asymp\omega  \in  \Sigma(I)$. If $\eta \notin \ell_1^*$, assume by
contradiction that $\eta_a \notin \Sigma(I)$. As in the proof of
Lemma~\ref{lem:2.8}, choose a sequence of generators
$\rho^{(k)}$ for $\Sigma(I)$ with $\rho^{(k)} \leq \rho^{(k+1)}$ and
such that for every $\xi  \in \Sigma(I)$ there is an $m \in
\mathbb{N}$ for which $\xi  = O(\rho^{(m)})$. Then there is an
increasing sequence of indices $n_k$ such that
$(\frac{\eta_a}{\rho^{(k)}})_{n_k}\geq k$ for every $k$. Exploiting
the non-summability of $\eta$ we can further require that
$\frac{1}{n_k-n_{k-1}}\sum_{i=n_{k-1}+1}^{n_k}\eta_i\geq \frac12
(\eta_a)_{n_k}$ for every $k$. Set $n_o := 0$ and define $\xi_ j =
(\eta _a)_{n_k}$ for $n_{k-1} < j \leq n_k$. We prove by induction
that $(\xi_a)_j \leq (2\eta_a)_j$. The inequality holds trivially
for $j \leq n_1$ and assume it holds also for all $j \leq n_{k-1}$.
If $n_{k-1} < j \leq n_k$, it follows that
\begin{align*}
\sum^j_{i=1}\xi_ i&= n_{k-1}(\xi_a)n_{k-1} + (j - n_{k-1})
(\eta_a)_{n_k}
\\
&\leq 2 n_{k-1}(\eta_a)_{n_{k-1}} + (j - n_{k-1})(\eta_a)_{n_k}
\\
&\leq 2 \sum^{n_{k-1}}_{i=1}\eta_ i + 2
\frac{j-n_{k-1}}{n_k-n_{k-1}}\sum^{n_k}_{i=n_{k-1}+1}\eta_ i \\
&\leq 2 \sum^{n_{k-1}}_{i=1}\eta_ i + 2 \sum^{j}_{i=n_{k-1}+1}\eta_i 
=2j(\eta_a)_j
\end{align*}
where the last inequality follows because
$\frac{1}{j-n}\sum_{i=n+1}^j\eta_i$ is monotone nonincreasing in $j$
for $j > n$. Thus $\xi  \in  \Sigma((\eta)^-) \subset  \Sigma(I)$.
On the other hand, for every $m \in  \mathbb{N}$ and $k \geq m$, 
$(\frac{\xi}{\rho^{(m)} })_{n_k} \geq (\frac{\xi}{\rho^{(k)} })_{n_k} = (\frac{\eta_a}{\rho^{(k)}})_{n_k}\geq k$ and thus $\xi  \notin \Sigma(I)$, a contradiction.
\end{proof}

By Theorem~\ref{thm:2.5}, $I_- + J_-$ is am-closed for any pair of
ideal $I$ and $J$ and it is contained in $I+J$. Hence $I_- + J_-
\subset  (I+J)_-$ and this inclusion can be proper by
Theorem~\ref{thm:2.9} and Example~\ref{ex:2.4}(ii).

\begin{corollary} \label{cor:2.10}
If $I$ is a countably generated ideal, then $I_a$ is the smallest countably generated ideal containing $I^-$.
\end{corollary}

\begin{proof}
By the five chain inclusion, $I^- \subset  I_a$ and if $I^- \subset
J$ for some countably generated ideal $J$, then $I^- \subset  J_- =
\text{$_a$}J$ and hence $I_a = (I^-)_a \subset  J^o\subset  J$.
\end{proof}

As a consequence of Theorem~\ref{thm:2.9} we obtain also an
elementary proof of the following, which was obtained for the
principal ideal case by \cite[Theorem 3.11]{2}.
\pagebreak
\begin{theorem} \label{thm:2.11}
A countably generated ideal is am-closed if and only if it is
am-stable.
\end{theorem}

\begin{proof}
If $I$ is a countably generated am-closed ideal, then $I = I_-$ and
hence $I = \text{$_a$}I$ by Theorem~\ref{thm:2.9}, i.e., $I$ is
am-stable. On the other hand, every am-stable ideal is am-closed by
the five chain inclusion.
\end{proof}

\noindent $\mathcal{L}_1$ is an example of a non countably generated ideal which is am-closed
(and also am-$\infty$  closed) but is neither am-stable nor am-$\infty$  stable. \vspace{12pt}

Now we pass to the second main result of this section, namely that
the intersection of am-open ideals is am-open
(Theorem~\ref{thm:2.17}). To prove it and to provide tools for our
study in Section~\ref{sec:6} of the commutation relations between
the se and sc operations and the am-interior operation, we need the
characterization of the am-interior $I^o$ of an ideal $I$ given in
Corollary~\ref{cor:2.16} below. This in turn will lead naturally to
a characterization of the smallest am-open ideal $I^{oo}$ containing
$I$ (Definition~\ref{def:2.18} and Proposition~\ref{prop:2.21}).
Both characterizations depend on the principal ideal case.

As recalled earlier, an ideal $I$ is am-open if $I = J_a$ for some ideal $J$
(e.g., $J = I^- = \text{$_a$}(I_a))$. 
In terms of sequences, $I$ is am-open if and only if for every $\xi  \in \Sigma(I)$, one has $\xi \leq \eta_a \in  \Sigma(I)$ for some $\eta \in  c_o^*$. 
Remark~\ref{rem:2.15}(iii) show that there is a minimal $\eta_a \geq \xi$. 
First we note  when a sequence is equal to the arithmetic mean of a $c_o^*$-sequence. 
The proof is elementary and is left to the reader.

\begin{lemma} \label{lem:2.12}
A sequence $\xi$  is the arithmetic mean $\eta_a$ of some sequence
$\eta \in  c_o^*$ if and only if $0 \leq \xi  \to 0$ and
$\frac{\xi}{\omega}$  is monotone nondecreasing and concave, i.e.,
$(\frac{\xi}{\omega})_{n+1} \geq \frac12 ((\frac{\xi}{\omega})_n +
(\frac{\xi}{\omega})_{n+2})$ for all $n \in \mathbb{N}$ and $\xi_ 1
= (\frac{\xi}{\omega})_1 \geq \frac12 (\frac{\xi}{\omega})_2$.
\end{lemma}

It is elementary to see that for every $\eta \in c_o^*$, 
$(\eta)_a = (\eta_ a)$ and that $\eta_a$ satisfies the 
$\Delta_{1/2}$-condition because $\eta_a \leq D_m\eta_a \leq m\eta_a$ for every $m \in \mathbb{N}$.
In particular,
all the generators of the principal ideal $(\eta_ a)$
are equivalent.

\begin{lemma} \label{lem:2.13}
If $I$ is a principal ideal, then the following are equivalent.
\item[(i)]
$I$ is am-open

\item[(ii)]
$I = (\eta_ a)$ for some $\eta \in  c_o^*$

\item[(iii)]
$I = (\xi )$ for some $\xi  \in  c_o^*$ for which
$\frac{\xi}{\omega}$  is monotone nondecreasing.
\end{lemma}

\begin{proof}
(i) $\Leftrightarrow$ (ii). Assume that $I = (\xi )$ for some $\xi
\in  c_o^*$ and that $I$ is am-open and hence $I = J_a$ for some
ideal $J$. Then $\xi  \leq \eta_a$ for some $\eta_a \in \Sigma(I)$
and hence $\eta_a \leq MD_m\xi$  for some $M > 0$ and $m \in
\mathbb{N}$. Since $\eta_a \asymp\ D_m\eta_a$, it
follows that $\xi\asymp \eta_a$ and hence (ii)
holds. The converse holds since $(\eta_a) = (\eta)_a$.

(ii) $\Rightarrow$ (iii) is obvious.

(iii) $\Rightarrow$ (ii). $\frac{\xi}{\omega}$  is quasiconcave,
i.e., $\frac{\xi}{\omega}$  is monotone nondecreasing and $\omega \frac{\xi}{\omega}$ is monotone nonincreasing. 
Adapting to sequences the proof of Proposition~5.10 in Chapter~2 of \cite{4}
(see also \cite[Section 2.18]{7}) shows that if $\psi$ is the
smallest concave sequence that majorizes $\frac{\xi}{\omega}$, then
$\frac{\xi}{\omega} \leq \psi \leq 2 \frac{\xi}{\omega}$ and hence
$\psi  \asymp \frac{\xi}{\omega}$. Moreover, $\psi_1 =
\xi_ 1 = (\frac{\xi}{\omega})_1$ since otherwise we could lower
$\psi_1$ and still maintain the concavity of $\psi$. And since the
sequence  $\frac{\xi_1}{\omega}$  is concave and 
$\frac{\xi_1}{\omega}\geq  \frac{\xi}{\omega}$ , it follows by the minimality of $\psi$ that $\psi \leq
\frac{\xi_1}{\omega}$ and so, in particular, $\psi_1= \xi_1 \geq
\frac12 \psi_2$. Since $\psi$ is concave and nonnegative, it follows
that it is monotone nondecreasing.
But then, by Lemma~\ref{lem:2.12}
applied to $\omega\psi$, one has $\omega\psi = \eta_a$ for some
sequence $\eta \in  c_o^*$ and thus $(\xi ) = (\omega\psi) = (\eta_a)$.
\end{proof}

We need now the following notations from \cite[Section 2.3]{7}. 
The upper and lower monotone nondecreasing and monotone nonincreasing envelopes of a real-valued sequence $\phi$ are:
\[
\und  \phi :=\! \Big\langle\max_{ i\leq n} \phi_i\Big\rangle, \quad
\lnd \phi :=\! \Big\langle \inf_{i\geq n} \phi_i\Big\rangle, \quad
\uni \phi :=\! \Big\langle \sup_{i\geq n} \phi_i\Big\rangle, \quad
\lni \phi :=\! \Big\langle \min_{i\leq n} \phi_i\Big\rangle.
\]

\begin{lemma} \label{lem:2.14}
For every $\xi  \in  c_o^*$:
\item[(i)]
$(\xi )^o = (\omega \lnd \xo)$

\item[(ii)]
$(\omega  \und \xo )$ is the smallest am-open ideal containing $(\xi )$.
\end{lemma}

\begin{proof}
(i) We first prove that $\omega\lnd\xo$  is monotone nonincreasing.
Indeed, in the case when $(\lnd \xo)_n = (\lnd\xo)_{n+1}$, then $(\omega
\lnd\xo)_{n+1} \leq (\omega  \lnd\xo)_n$, but if on the other hand
$(\lnd\xo)_n \neq (\lnd\xo)_{n+1}$, then $(\lnd\xo)_n = (\xo)_n$ and
hence also
\[
(\omega  \lnd\xo)_{n+1} \leq \xi_{n+1} \leq \xi_n =
(\omega  \lnd\xo)_n.
\]
Moreover, $\omega \lnd\xo \to 0$ since $\omega
\lnd\xo\leq \xi$. Thus $(\omega  \lnd\xo)\subset (\xi )$. By
Lemma~\ref{lem:2.13}(i) and (iii), $(\omega  \lnd\xo)$ is am-open
and hence $(\omega  \lnd\xo)\subset  (\xi )^o$. For the reverse
inclusion, if $\mu  \in  \Sigma((\xi )^o)$, then $\mu  \leq \zeta_a$
for some $\zeta_a\in  \Sigma(\xi )$, i.e., $\zeta_a\leq MD_m\xi$ for
some $M > 0$ and $m \in \mathbb{N}$. Then $D_m \zeta_a\leq m
\zeta_a\leq mMD_m \xi$, whence $\frac{\zeta_a}{\omega} \leq mM\xo$.
As $\frac{\zeta_a}{\omega}$ is monotone nondecreasing, also
$\frac{\zeta_a}{\omega}\leq mM \lnd\xo$ so that $\mu \leq mM\omega
\lnd\xo$. Thus $(\xi )^o \subset  (\omega  \lnd\xo)$.

(ii) A similar proof as in (i) shows that $\omega  \und \xo\in
c_o^*$. Since by definition \linebreak $\xi  \leq \omega  \und \xo$, we have
that $(\xi ) \subset  (\omega  \und \xo)$, and the latter ideal is
am-open by Lemma~\ref{lem:2.13}. If $I$ is any am-open ideal
containing $(\xi)$, then $\xi  \leq \zeta_a$ for some $\zeta_a\in
\Sigma(I)$ and again, since $\frac{\zeta_a}{\omega}$ is monotone
nondecreasing, $\omega  \und \xo \leq\zeta_a$, hence $(\omega  \und
\xo) \subset  I$.
\end{proof}

\begin{remark} \label{rem:2.15}

(i) Lemma~\ref{lem:2.14}(i) shows that the am-interior $(\xi )^o$
of a principal ideal $(\xi)$ is always principal and its generator
$\omega \lnd\xo$ is unique up to equivalence by Lemma~\ref{lem:2.13}
and preceding remarks. Notice that $(\omega)$ being the smallest
nonzero am-open ideal, $(\xi )^o = \{0\}$ if and only if $(\omega)
\not\subset (\xi )$. In terms of sequences, this corresponds to the
fact that that $\lnd\xo = 0$ if and only if $(\omega) \not\subset
(\xi )$.

(ii) While $(\omega  \lnd\xo)$ is the largest am-open ideal
contained in $(\xi )$ by
Lemma \ref{lem:2.14}(i), it is easy to see
that there is no (pointwise) nonzero largest arithmetic mean
sequence majorized by $\xi$  unless $\xi$  is itself an arithmetic
mean. However, there is an arithmetic mean sequence $\eta_a$
majorized by $\xi$  which is the largest in the O-sense (actually up
to a factor of 2). Indeed, let $\psi$ be the smallest concave
sequence that majorizes the quasiconcave sequence $\frac12\lnd\xo$.
Then, as in the proof of Lemma~\ref{lem:2.13}(iii) $\Rightarrow$
(ii), $\psi =\frac{\eta_a}{\omega}$ for some $\eta \in  c_o^*$ and
$\psi \leq \lnd\xo$ and hence $\eta_a \leq \xi$. Moreover, for every
$\rho \in c_o^*$ with $\rho_a \leq \xi$, since
$\frac{\rho_a}{\omega}$  is monotone nondecreasing, it follows that
$\frac{\rho_a}{\omega} \leq \lnd\xo \leq 2 \frac{\eta_a}{\omega}$
and hence $\rho_a \leq 2\eta_a$.

(iii) Lemma~\ref{lem:2.14}(ii) shows that $(\omega  \und \xo)$ is
the smallest am-open ideal containing $(\xi )$, and moreover, from the proof
of Lemma~\ref{lem:2.13}(iii) we see that $(\omega  \und \xo) =
(\eta_a)$ where $\frac{\eta_a}{\omega}$  is the smallest concave
sequence that majorizes the quasiconcave sequence $\und \xo$. In
contrast to (ii), $\eta_a$ is also the (pointwise) smallest
arithmetic mean that majorizes $\xi$. Indeed, if $\rho_a \geq \xi$
then  $\frac{\rho_a}{\omega}\geq \und \xo$ because
$\frac{\rho_a}{\omega}$ is monotone nondecreasing and moreover
$\frac{\rho_a}{\omega}\geq \frac{\eta_a}{\omega}$ because
$\frac{\rho_a}{\omega}$ is concave.

(iv) By \cite[Section 2.33]{7}, $\omega  \lnd\xo$ is the reciprocal
of the fundamental sequence of the Marcinkiewicz norm for $_a(\xi
)$.
\end{remark}

\begin{corollary} \label{cor:2.16}

For every ideal $I$:

\item[(i)]
$\Sigma(I^o) = \{\xi  \in  c_o^* \mid \omega  \und \xo \in
\Sigma(I)\} = \{\xi  \in  c_o^* \mid \xi  \leq \omega  \lnd
\frac{\eta}{\omega}$ for some $\eta \in  \Sigma(I)\}$.

\item[(ii)]
If $I$ is an am-open ideal, then $\xi  \in  \Sigma(I)$ if and only
if $\omega  \und \xo \in  \Sigma(I)$.

\end{corollary}

\begin{proof}
If $\xi  \in  \Sigma(I^o)$, then $(\xi ) \subset  (\omega  \und \xo)
\subset  I^o$ by Lemma~\ref{lem:2.14}(ii), whence $\omega  \und \xo
\in  \Sigma(I)$. If $\omega  \und \xo \in  \Sigma(I)$, then $\xi
\leq \omega \und \frac{\xi}{\omega} = \omega\lnd(\frac{\omega\und
\xo}{\omega})$. Finally, if $\xi  \leq \omega  \lnd
\frac{\eta}{\omega}$ for some $\eta \in  \Sigma(I)$, then $\omega
\lnd\frac{\eta}{\omega}  \in \Sigma((\eta)^o)\subset \Sigma(I^o)$ by
Lemma~\ref{lem:2.14}(i) and hence $\xi  \in  \Sigma(I^o)$. Thus all
three sets are equal. This proves (i) and (ii) is a particular case.
\end{proof}

An immediate consequence of Corollary~\ref{cor:2.16}(ii) is the
following result.

\begin{theorem}\label{thm:2.17} Intersections of am-open ideals are
am-open.
\end{theorem}

Since $I\subset  I_a$, the collection of all am-open ideals
containing $I$ is always nonempty. By Theorem~\ref{thm:2.17} its
intersection is am-open, hence it is the smallest am-open ideal
containing $I$.

\begin{deff}\label{def:2.18} For each ideal $I$, denote $I^{oo}:=
\bigcap\{J \mid J \supset I \text{ and $J$ is am-open} \}$.
\end{deff}

\begin{remark} \label{rem:2.19}
Lemma~\ref{lem:2.14} affirms that if $I$ is principal, so are $I^o$
and $I^{oo}$.
\end{remark}

Notice that $I^o \subset I \subset  I^{oo}$ and $I$ is am-open if
and only if one of the inclusions and hence both of them are
equalities. Since $I \subset  I_a$ and $I_a$ is am-open, $I^{oo}
\subset  I_a$. The inclusion can be proper even for principal
ideals. Indeed if $\xi  \in  c_o^*$ and $\xi_a$ is irregular, i.e.,
$\xi_{a^ 2} \neq O(\xi_a)$, then $I = (\xi_a)$ is am-open and hence $I
= I^{oo}$, but $I_a = (\xi_{a^ 2} ) \neq (\xi_a ) = I^{oo}$. Of course,
if $I$ is am-stable then $I = I^{oo} = I_a$, and if $\{0\} \neq I
\subset  \mathcal{L}_1$ then $(\omega) = I^{oo} = I_a$, but as the
following example shows, the equality $I^{oo} = I_a$ can hold also
in other cases.

\begin{example} \label{ex:2.20}

Let $\xi_ j = \frac{1}{k!}$  for $((k-1)!)^2 < j \leq (k!)^2$. Then
direct computations show that $\xi$  is irregular, indeed does not
even satisfy the $\Delta_{1/2}$-condition, is not summable, but
$\xi_a = O(\omega  \und \xo)$ and hence by Lemma~\ref{lem:2.14}
(ii), $(\xi )^{oo} = (\xi )_a$.
\end{example}

The characterization of $I^{oo} = (\omega  \und
\frac{\eta}{\omega})$ provided by Lemma~\ref{lem:2.14}(ii) for
principal ideals $I = (\xi )$ extends to general ideals.

\begin{proposition} \label{prop:2.21}

For every ideal $I$, the characteristic set of $I^{oo}$ is given by
\[
\Sigma(I^{oo}) = \bigg\{\xi  \in  c_o^* \mid \xi  \leq \omega  \und \frac{\eta}{\omega} \text{ for some  } \eta \in  \Sigma(I)\bigg\}.
\]
\end{proposition}

\begin{proof}
Let $\Sigma = \{\xi  \in  c_o^* \mid \xi  \leq \omega  \und \frac{\eta}{\omega}$ for some $\eta \in  \Sigma(I)\}$. 
First we show that $\Sigma$ is a characteristic set. Let $\xi, \rho \in  \Sigma$,
i.e., $\xi  \leq \omega  \und \frac{\eta}{\omega}$ and $\rho \leq \omega  \und \frac{\mu}{\omega}$ for some $\eta, \mu  \in \Sigma(I)$. 
Since $\omega  \und \frac{\eta}{\omega}+ \omega  \und \frac{\mu}{\omega}\leq 2\omega\frac{\eta+\mu}{\omega}$ and $\eta + \mu  \in  \Sigma(I)$, 
it follows that $\xi  + \rho \in  \Sigma$.
Moreover, if $\xi  \leq \omega  \und \frac{\eta}{\omega}$, then for all $m$,
\[
D_m\xi \leq D_m\omega D_m \und \frac{\eta}{\omega} = D_m\omega \und
D_m \frac{\eta}{\omega} \leq m\omega \und  \frac{D_m\eta}{\omega}
\]
and hence $D_m\xi  \in  \Sigma$,  i.e., $\Sigma$ is closed under
ampliations. Clearly, $\Sigma$ is also closed under multiplication
by positive scalars and it is hereditary. Thus $\Sigma$ is a
characteristic set and hence $\Sigma = \Sigma(J)$ for some ideal
$J$. Then $J \supset I$ follows from the inequality $\xi \leq \omega
\und \xo$. If $\eta \in  \Sigma(J)$, i.e., $\eta\leq \omega  \und
\xo$ for some $\xi  \in  \Sigma(I)$, then also $\omega  \und
\frac{\eta}{\omega} \leq \omega  \und \xo$ and hence $\omega  \und
\frac{\eta}{\omega} \in  \Sigma(J)$. By Corollary~\ref{cor:2.16},
this implies that $J$ is am-open and hence  $J \supset I^{oo} $. For
the reverse inclusion, if $\eta \in  \Sigma(J)$,
i.e., $\eta \leq
\omega  \und \xo$ for some $\xi  \in  \Sigma(I)$, then  $\omega \und
\xo \in  \Sigma((\xi )^{oo}) \subset  \Sigma(I^{oo})$ by
Lemma~\ref{lem:2.14}(ii). Thus $\eta \in  \Sigma(I^{oo})$, hence $J
\subset  I^{oo}$, and we have equality.
\end{proof}

\noindent As a consequence of this proposition and by the subadditivity of
``$\und$'', we see that $(I + J)^{oo} = I^{oo} + J^{oo}$ for any two
ideals $I$ and $J$.

For completeness' sake we collect in the following lemma the
distributivity properties of the $I^{oo}$ and $I_-$ operations.

\begin{lemma} \label{lem:2.22}
For all ideals $I$, $J$:

\item[(i)]
$I^{oo} + J^{oo} = (I + J)^{oo}$ (paragraph after
Proposition~\ref{prop:2.21})

\item[(ii)]
$I_-+J_- \subset  (I+J)_-$ and the inclusion can be proper (remarks after Theorem~\ref{thm:2.9}).

Let $\{I_\gamma , \gamma \in  \Gamma\}$ be a collection of ideals. Then

\item[(iii)]
$(\bigcap_\gamma I_\gamma)^{oo} \subset \bigcap_\gamma (I_\gamma)^{oo}$
(the inclusion can be proper by Example~\ref{ex:2.23}(i))

\item[(iv)]
$(\bigcap_\gamma I_\gamma)_- = \bigcap_\gamma (I_\gamma)_-$
(by Lemma~\ref{lem:2.2}(v))

If $\{I_\gamma , \gamma \in  \Gamma\}$ is directed by inclusion, then

\item[(v)]
$(\bigcup_\gamma I_\gamma)^{oo} = \bigcup_\gamma (I_\gamma)^{oo}$
(by Lemma~\ref{lem:2.1}(v))

\item[(vi)]
$(\bigcup_\gamma I_\gamma)_- \supset \bigcup_\gamma (I_\gamma)_-$
(the inclusion can be proper by Example~\ref{ex:2.23}(ii))
\end{lemma}

\begin{example} \label{ex:2.23}

\item[(i)]
The inclusion in (iii) can be proper even if $\Gamma$ is finite.
Indeed for the same construction as in Example~\ref{ex:2.4}(i),
$((\xi ) \cap (\eta))^{oo} = (\min(\xi , \eta))^{oo} = (\omega)$
since $\min (\xi ,\eta)$ is summable, while $\omega = o(\omega  \und
\xo)$ and $\omega = o(\omega  \und \frac{\eta}{\omega})$ since $\xi$
and $\eta$  are not summable. Thus $\omega = o(\min(\omega  \und \xo
,\omega\und \frac{\eta}{\omega}))$ and hence
\[
(\omega) \not\subset  \bigg(\omega  \und \xo\bigg) \cap \bigg(\omega
\und \frac{\eta}{\omega}\bigg)
=
(\xi )^{oo} \cap (\eta)^{oo}
\]

\item[(ii)]
The inclusion in (vi) can be proper. $\mathcal{L}_1$ as every ideal with the
exception of $\{0\}$ and $F$, is the directed union of distinct ideals
$I_\gamma$. Since $\mathcal{L}_1$ is the smallest am-closed ideal, $(I_\gamma )_- = \{0\}$ for
every $\gamma$. Thus $\mathcal{L}_1 = (\bigcup_\gamma   I_\gamma )_-$ while
$\bigcup_\gamma  (I_\gamma )_-= \{0\}$.
\end{example}

\section{Arithmetic Mean Ideals at Infinity}\label{sec:3}

The arithmetic mean is not adequate for distinguishing between
nonzero ideals contained in the trace-class since they all have the
same arithmetic mean $(\omega)$ and the same pre-arithmetic mean
$\{0\}$. The ``right''  tool for  ideals in the trace-class is
the arithmetic mean at infinity which was employed for sequences in
\cite{1, 7, 13, 21} among others.  For every summable sequence $\eta$,
\[
\eta_{a_\infty} := \langle \frac1n \sum^\infty_{n+1}\eta_j\rangle.
\]
Many of the properties of the arithmetic mean and of the am-ideals
have a dual form for the arithmetic mean at infinity but there are
also substantial differences often linked to the fact that contrary
to the am-case, the arithmetic mean at infinity $\xi_{a_\infty}$ of
a sequence $\xi  \in  \ell_1^*$  may fail to satisfy the
$\Delta_{1/2}$ condition and also may fail to majorize $\xi$  (in
fact, $\xi_{a_\infty}$ satisfies the $\Delta_{1/2}$ condition if and
only if $\xi  = O(\xi_{a_\infty})$, see \cite[Corollary~4.4]{10}).
Consequently the results and proofs tend to be harder.

In \cite{10} we defined for every ideal $I \neq \{0\}$ the
am-$\infty$ ideals $_{a_\infty} I$ (pre-arithmetic mean at infinity)
and $I_{a_\infty}$ (arithmetic mean at infinity) with characteristic
sets:
\begin{align*}
\Sigma(_{a_\infty} I) &= \{\xi  \in  \ell_1^*  \mid \xi_{a_\infty}
\in  \Sigma(I)\}
\\
\Sigma(I_{a_\infty}) &= \{\xi  \in  c_o^* \mid \xi  = O(\eta_{a_\infty})
\text{ for some } \eta \in  \Sigma(I \cap \mathcal{L}_1)\}
\end{align*}
Notice that $\xi_{a_\infty}= o(\omega)$ for all  $\xi  \in  \ell_1^*$.
Let $\se (\omega)$ denote the ideal with characteristic set $\{\xi
\in  c_o^* \mid \xi  = o(\omega)\}$ (see Definition~\ref{def:4.1}
for the soft-interior $\se I$ of a general ideal $I$). Thus
\[
_{a_\infty} I = \,_{a_\infty} (I \cap \se (\omega)) \subset
\mathcal{L}_1 \qquad \text{and}\qquad I_{a_\infty} = (I \cap
\mathcal{L}_1)_{a_\infty} \subset  \se (\omega).
\]
In \cite[Corollary~4.10]{10} we defined an ideal $I$ to be am-$\infty$
stable if $I = \text{$_{a_\infty}$}I$
(or, equivalently, if $I\subset
\mathcal{L}_1$ and $I = I_{a_\infty})$. There is a largest
am-$\infty$  stable ideal, namely the lower stabilizer at infinity of 
$\mathcal{L}_1$, $st_{a_\infty} (\mathcal{L}_1)
=\bigcap^\infty_{n=0}{} _{a_\infty^n}(\mathcal{L}_1)$, which
together with the smallest nonzero am-stable ideal
$st^a(\mathcal{L}_1)$ defined earlier plays an important role in
\cite{10}.

Natural analogs to the am-interior and am-closure are the
am-$\infty$  interior  of an ideal $I$
\[
I^{o\infty} := (_{a_\infty} I)_{a_\infty}= (I
\cap \se (\omega))^{o\infty}
\]
  and the am-$\infty$  closure  of an ideal $I$
\[
I^{-\infty} := \,_{a_\infty}(I_{a_\infty}) = (I \cap
\mathcal{L}_1)^{-\infty}.
\]
 We call an ideal $I$
am-$\infty$  open (resp., am-$\infty$  closed) if $I = I^{o\infty}$
(resp., $I = I^{-\infty}$).

In \cite[Proposition~4.8]{10} we proved the analogs of the 5-chain
of inclusions for am-ideals (see Section~\ref{sec:2} paragraph 5 and
\cite[Section 2]{10}):
\[
_{a_\infty}I \subset  I^{o\infty}  \subset I \cap \se (\omega)
\]
and
\[
I \cap \mathcal{L}_1 \subset  I^{-\infty}  \subset  I_{a_\infty}\cap
\mathcal{L}_1
\]
and the idempotence of the maps $I\to  I^{o\infty}$  and $I\to
I^{-\infty}$, a consequence of the more general identities
\[
_{a_\infty} I = \,_{a_\infty}((_{a_\infty}I)_{a_\infty})\qquad
\text{and}\qquad I_{a_\infty}=
(_{a_\infty}(I_{a_\infty}))_{a_\infty}.
\]
Thus, like in the am-case, an ideal $I$ is am-open (resp.,
am-$\infty$ closed) if and only if there is an ideal $J$ such that
$I = J_{a_\infty}$ (resp., $I = \text{$_{a_\infty}$}J$). As
$(\mathcal{L}_1)_{a_\infty}  = \se (\omega)$ and $\,_{a_\infty}  \se
(\omega) = \mathcal{L}_1$ (see \cite[Lemma~4.7, Corollary~4.9]{10}), $\se
(\omega)$ and $\mathcal{L}_1$ are, respectively, the largest
am-$\infty$  open and the largest am-$\infty$  closed ideals. The
finite rank ideal $F$ is am-$\infty$  stable and hence it is the smallest nonzero am-$\infty$  open ideal
and the smallest nonzero am-$\infty$  closed ideal. Moreover, every
nonzero ideal with the exception of $F$ contains a nonzero principal
am-$\infty$  stable ideal (hence both am-$\infty$  open and
am-$\infty$  closed) distinct from $F$ \cite{12}. Contrasting these
properties for the am-$\infty$  case with the properties for the am
case, $(\omega)$ is the smallest nonzero am-open ideal, while $\mathcal{L}_1$ is
the smallest nonzero am-closed ideal, and every principal ideal is
contained in an am-stable principal ideal (hence both am-open and
am-closed) and so there are no proper largest am-closed or am-open
ideals.

We leave to the reader to verify that the exact analogs of Lemmas \ref{lem:2.1}, \ref{lem:2.2} and \ref{lem:2.3} hold for the am-$\infty$  case. 
Here Theorem \ref{thm:3.2} plays the role of Theorem \ref{thm:2.5} for the equality in Lemma \ref{lem:2.3}(iv)
and Theorem \ref{thm:3.11} plays the role of Theorem \ref{thm:2.17} for the equality in Lemma \ref{lem:2.2}(iii). 
The same counterexample to equality in Lemma \ref{lem:2.2}. 
(ii) given in Example \ref{ex:2.4}(i) provides a counterexample to the equality in the analog am-$\infty$ case: by \cite[Lemma~4.7]{10}, 
$((\xi ) \cap (\eta))_{a_\infty} = (\min(\xi , \eta))_{a_\infty} = ((\min(\xi , \eta))_{a_\infty})$ 
while $(\xi )_{a_\infty} = (\eta)_{a_\infty} = \se(\omega)$. 
The counterexample to the equality in Lemma \ref{lem:2.3}(iii) and hence (ii) given in
Example \ref{ex:2.4}(ii) provides also a counterexample to the same
equalities in the am-$\infty$  analogs, but we postpone verifying
that until after Lemma \ref{lem:3.9}.

The distributivity of the am-$\infty$  closure over finite sums,
i.e., the am-$\infty$ 
analog of Theorem~\ref{thm:2.5}, also holds,
but for its proof we no longer can depend on the theory of
substochastic matrices. Instead we will use the following finite
dimensional lemma and then we will extend it to the infinite
dimensional case via the $w^*$ compactness of the unit ball of
$\ell_1$.

\begin{lemma} \label{lem:3.1}
Let  $\xi, \eta$, and $\mu  \in  [0, \infty)^n$ for some $n \in
\mathbb{N}$.  If for all $1 \leq k \leq n$, $\sum_{j=1}^k \eta_j + \sum_{j=1}^k \mu_j \leq
\sum_{j=1}^k \xi_j$,  then there exist
$\tilde{\eta}$ and $\tilde{\mu} \in  [0, \infty)^n$ for which $\xi
= \tilde{\eta} + \tilde{\mu}$, $\sum_{j=1}^k \eta_j \leq \sum_{j=1}^k
\tilde{\eta}_j $, and $\sum_{j=1}^k \mu_j \leq \sum_{j=1}^k
\tilde{\mu}_j$ for all $1\leq k \leq n$.
\end{lemma}

\begin{proof}
The proof is by induction on $n$. The case $n = 1$ is trivial, so
assume the property is true for all integers less than equal to $n
-1$. Assume without loss of generality that $\sum_{j=1}^k \xi_j> 0$
for all $1 \leq k \leq n$ and let
\[
\gamma   = \max_{1\leq k\leq n}
\frac{\sum_{j=1}^k \eta_j + \sum_{j=1}^k \mu_j}{\sum_{j=1}^k \xi_j},
\]
which maximum $\gamma \leq 1$ is achieved for some $k$. Then
\[
\sum_{j=1}^m \eta_j + \sum_{j=1}^m \mu_j \leq \gamma \sum_{j=1}^m
\xi_j \quad\text{for all } 1 \leq m \leq k,
\]
with equality holding for $m = k$, so also 
\[
 \sum_{j=k+1}^m \eta_j +\sum_{j=k+1}^m \mu_j\leq \gamma
\sum_{j=k+1}^m \xi_j \quad \text{for all } k +1 \leq m \leq
n.
\]
Thus if we apply the induction hypothesis
separately to the truncated sequences $\gamma  \xi \chi_{[1, k]}$,
$\eta \chi_{[1, k]}$ and $\mu \chi_{[1, k]}$ and to $\gamma  \xi
\chi_{[k+1, n]}$, $\eta \chi_{[k+1, n]}$, and $\mu \chi_{[k+1, n]}$
we obtain that $\gamma  \xi \chi_{[1, k]} = \rho + \sigma$ for two sequences
$\rho, \sigma \in  [0, \infty)^k$ for which $\sum_{j=1}^m \eta_j \leq
\sum_{j=1}^m \rho_j$ and $\sum_{j=1}^m \mu_j \leq \sum_{j=1}^m
\sigma_j$  for all $1\leq m \leq k$. Similarly $\gamma  \xi
\chi_{[k+1, n]} = \rho' + \sigma' $ for two sequence $ \rho' , \sigma' \in  [0, \infty)^{n-k}$ and
$\sum_{j=k+1}^m \eta_j \leq \sum_{j=k+1}^m \rho'_j$, $\sum_{j=k+1}^m
\mu_j \leq \sum_{j=k+1}^m \sigma'_j$ for all $k+1\leq m \leq n$. But
then it is enough to define $\tilde{\eta} = \frac{1}{\gamma} \langle
\rho, \rho'\rangle$ and $\tilde{\mu} = \frac{1}{\gamma} \langle
\rho, \rho'\rangle$ and verify that it satisfies the required
condition.
\end{proof}

\begin{theorem} \label{thm:3.2}
$(I + J)^{-\infty}  = I^{-\infty} + J^{-\infty}$ for all ideals $I$, $J$. 
 
In particular, the sum of two am-$\infty$ closed ideals is am-$\infty$ closed.
\end{theorem}

\begin{proof}
Let $\xi  \in  \Sigma((I + J)^{-\infty})$, i.e., $\xi_{a_\infty} \leq (\eta  + \mu )_{a_\infty}  = \eta_{a_\infty}  + \mu_{a_\infty}$
for some $\eta \in  \Sigma(I \cap \mathcal{L}_1)$ and $\mu  \in \Sigma(J \cap \mathcal{L}_1)$. 
By increasing if necessary the values of $\xi_1$ or $\eta_1$, 
we can assume that
$\sum_{j=1}^\infty \xi_j= \sum_{j=1}^\infty \eta_j+ \sum_{j=1}^\infty \mu_j$ and hence $\eta_a + \mu_a \leq \xi_a$. 
By applying Lemma~\ref{lem:3.1} to the truncated sequences $\xi \chi_{[1, n]}$, $\eta \chi_{[1, n]}$, and $\mu \chi_{[1, n]}$,
we obtain two sequences 
\[
\eta^{(n)}:=\langle \eta^{(n)}_1, \eta^{(n)}_2,
\ldots, \eta^{(n)}_n, 0,0, \ldots\rangle \quad \text{and} \quad  \mu^{(n)}:=\langle
\mu^{(n)}_1, \mu^{(n)}_2, \ldots, \mu^{(n)}_n, 0,0, \ldots\rangle
\]
for which $\xi_j = \eta_j^{(n)} + \mu_j^{(n)}$ for all $1 \leq j
\leq n$ and 
\[
\sum_{j=1}^m \eta_j \leq \sum_{j=1}^m \eta_j^{(n)} \quad \text{and} \quad
\sum_{j=1}^m \mu_j \leq \sum_{j=1}^m \mu_j^{(n)} \quad \text{ for all } m \leq
n.
\] 
Since $0 \leq \eta^{(n)}$ and $\mu^{(n)}\leq \xi$, by the
sequential compacteness of the unit ball of $\ell_1$ in the
$w^*$-topology (as dual of $c_o$), we can find converging
subsequences $\eta^{(n_k)}\underset{w^*}{\to} \tilde{\eta}$,
$\mu^{(n_k)}\underset{w^*}{\to}\tilde{\mu}$. It is now easy to
verify that $\xi  = \tilde{\eta} + \tilde{\mu}$, that
$\tilde{\eta}\geq 0$, $\tilde{\mu}\geq 0$, and that $\sum_{j=1}^n
\eta_j \leq \sum_{j=1}^n \tilde{\eta}_j$ and $\sum_{j=1}^n \mu_j
\leq \sum_{j=1}^n \tilde{\mu}_j$ for all $n$. 
It follows from
$\sum_{j=1}^\infty \xi_j= \sum_{j=1}^\infty \eta_j + \sum_{j=1}^\infty \mu_j$ 
 that  
$\sum_{j=1}^\infty
\tilde{\eta}_j = \sum_{j=1}^\infty \eta_j$ and $\sum_{j=1}^\infty
\tilde{\mu}_j = \sum_{j=1}^\infty \mu_j$, and hence 
$\sum_{j=n}^\infty \tilde{\eta}_j \leq \sum_{j=n}^\infty \eta_j$ and 
$\sum_{j=n}^\infty \tilde{\mu}_j \leq \sum_{j=n}^\infty \mu_j$ for all $n$. 
Let $\tilde{\eta}^*$, $\tilde{\mu}^*$ be the decreasing rearrangement of $\tilde{\eta}$ and $\tilde{\mu}$. 
Since $\sum_{j=n}^\infty \tilde{\eta}_j^* \leq \sum_{j=n}^\infty \tilde{\eta}_j$ for every $n$,
it follows that $(\tilde{\eta}^*)_{a_\infty} \leq \eta_{a_\infty}$, i.e., $\tilde{\eta}^* \in \Sigma(I^{-\infty})$. 
Thus $\tilde{\eta} \in  S(I^{-\infty})$.
Similarly, $\tilde{\mu} \in  S(J^{-\infty})$. 
But then $\xi  \in S(I^{-\infty}) + S(J^{-\infty}) = S(I^{-\infty} + J^{-\infty})$,
which proves that $\xi \in \Sigma (I^{-\infty} + J^{-\infty})$ and hence 
$(I + J)^{-\infty}  \subset  I^{-\infty}  + J^{-\infty}$. 
Since the am-$\infty$ closure
operation preserves inclusions, 
$I^{-\infty} + J^{-\infty} \subset (I + J)^{-\infty}$, 
concluding the proof.
\end{proof}

As a consequence, as in the am-case the collection of all the
am-$\infty$  closed ideals contained in an ideal $I$ is directed and
hence its union is an am-$\infty$  closed ideal by the am-$\infty$
analog of Lemma~\ref{lem:2.1}(v).

\begin{corollary} \label{cor:3.3} For every ideal $I$,
$
I_{-\infty}  := \bigcup\,\{J \mid J \subset I \text{ and } J \text{ is am-$\infty$ closed}\}
$
is the largest am-closed ideal contained in $I$.
\end{corollary}
\noindent Notice that $I_{-\infty}  \subset I \cap \mathcal{L}_1 \subset
I^{-\infty}$  and $I$ is am-$\infty$  closed if and only if
$I_{-\infty}  = I$ if and only if $I = I^{-\infty}$. Moreover,
$_{a_\infty} I$ is am-$\infty$  closed, so $_{a_\infty} I \subset
I_{-\infty}$. The inclusion can be proper: consider any ideal $I$
that is am-$\infty$  closed but not am-$\infty$ stable, e.g.,
$\mathcal{L}_1$. Analogously to the am-case, we can identify
$I_{-\infty}$ for $I$ countably generated.

\begin{theorem} \label{thm:3.4}
If $I$ is a countably generated ideal, then $I_{-\infty} = \text{$_{a_\infty}$}I$.
\end{theorem}

\begin{proof}
Let $\eta \in  \Sigma(I_{-\infty})$. Since $I_{-\infty} \subset \mathcal{L}_1$, the largest am-$\infty$ closed ideal, $\eta \in \ell_1^*$. 
We claim that $\eta_{a_\infty}  \in  \Sigma(I)$, i.e., $\eta \in  \Sigma(_{a_\infty} I)$. 
This will prove that $I_{-\infty} \subset  \text{$_{a_\infty}$}I$ and hence the equality. 
Assume by contradiction that $\eta_{a_\infty}  \notin \Sigma(I)$ and as in the proof of Lemma~\ref{lem:2.8}, 
choose a sequence of generators $\rho^{(k)}$ for $\Sigma(I)$ with $\rho^{(k)} \leq \rho^{(k+1)}$ and so that for every $\xi  \in  \Sigma(I)$, 
$\xi  = O(\rho^{(m)})$ for some $m \in \mathbb{N}$. 
Then there is an increasing sequence of indices $n_k$ such that $(\frac{\eta_{a_\infty}}{\rho^{(k)}})_{n_k}\geq k$ for every $k \in \mathbb{N}$. 
By the summability of $\eta$, we can further request that $\sum^{n_k}_{j=n_{k-1}+1}\eta_j \geq \frac12 \sum^{\infty}_{j=n_{k-1}+1}\eta_j$. 
Set $n_o := 0$ and define $\xi_j = (\eta _{a_\infty} )_{n_k}$ for $n_{k-1} < j \leq n_k$. 
Then
\begin{align*}
\sum_{i=j+1}^{\infty}\xi_i&=(n_k - j ) \xi_{n_k} +
\sum_{i=k+1}^{\infty} (n_i - n_{i-1})\xi_{n_i}\\
&= \frac{n_k-j}{n_k}\sum_{i=n_k+1}^{\infty} \eta_i +
\sum_{i=k+1}^{\infty} \frac{n_i - n_{i-1}}{n_i}
\sum_{m=n_i+1}^{\infty} \eta_m\\
&\leq \sum_{i=n_k+1}^{\infty} \eta_i + 2 \sum_{i=k+1}^{\infty}
\sum_{m=n_i+1}^{n_{i+1}}\eta_m\\
&\leq 3 \sum_{i=n_k+1}^{\infty} \eta_i
\leq 3 \sum_{i=j+1}^{\infty} \eta_i.
\end{align*}
Thus $\xi  \in  \Sigma((\eta)^{-\infty}) \subset \Sigma(I_{-\infty}
) \subset \Sigma(I)$. On the other hand, for every $m \in
\mathbb{N}$ and for every $k \geq m$,
$(\frac{\xi}{\rho^{(m)}})_{n_k} \geq (\frac{\xi}{\rho^{(k)}})_{n_k}
= (\frac{\eta_{a_\infty}}{\rho^{(k)}})_{n_k} \geq k$, whence $\xi
\not\in \Sigma(I)$, 
a contradiction.

\end{proof}

Precisely as for the am-case we have:

\begin{theorem}\label{thm:3.5}
A countably generated ideal is am-$\infty$ closed if and only if it
is am-$\infty$ stable.
\end{theorem}

Now we investigate the operations $I \rightarrow I^{o\infty}$  and
$I \rightarrow I^{oo\infty}$, where $I^{oo\infty}$ is the
am-$\infty$ analog of $I^{oo}$ and will be defined in
Definition~\ref{def:3.12}. While the statements are analogous to the
statements in Section~\ref{sec:2}, the proofs are sometimes
substantially different. The analog of Lemma~\ref{lem:2.12} is given
by:

\begin{lemma}\label{lem:3.6}
A sequence $\xi$ is the arithmetic mean at infinity
$\eta_{a_\infty}$ of some sequence $\eta \in  \ell_1^*$  if and only
if $\frac{\xi}{\omega} \in c_o^*$ and is convex, i.e.,
$(\frac{\xi}{\omega})_{n+1} \leq \frac12
((\frac{\xi}{\omega})_n+(\frac{\xi}{\omega})_{n+2})$ for all $n$.
\end{lemma}

The analog of Lemma~\ref{lem:2.13} is given by:

\begin{lemma}\label{lem:3.7} For every principal ideal $I$, the following are equivalent.
\item[(i)] $I$ is am-$\infty$ open.

\item[(ii)] $I = (\eta _{a_\infty})$ for some $\eta \in \ell_1^*$.

\item[(iii)] $I = (\xi )$ for some $\xi$ for which $\xo  \in c_o^*$.
\end{lemma}

\begin{proof}
(i) $\Leftrightarrow$ (ii). Assume $I$ is am-$\infty$ open and that
$\xi  \in  c_o^*$ is a generator of $I$. Then $I = J_{a_\infty}$ for
some ideal $J$, i.e., $\xi  \leq \eta_{a_\infty}$ for some $\eta \in
\ell_1^*$ such that $\eta_{a_\infty} \in  \Sigma(I)$ and thus $(\xi
) = (\eta_{a_\infty})$. The other implication is a direct
consequence of the equality $(\eta_{a_\infty} ) = (\eta)_{a_\infty}$
obtained in \cite[Lemma~4.7]{10}.

(ii) $\Rightarrow$ (iii). Obvious as $\frac{\eta_{a_\infty}}{\omega}
\downarrow 0$.

(iii) $\Rightarrow$ (ii). Since $F = (\langle 1, 1, 0, 0,
\ldots\rangle)_{a_\infty}$, we can assume without loss of generality
that $\xi_j > 0$ for all $j$. Let $\psi$ be the largest (pointwise)
convex sequence majorized by $\xo$. It is easy to see that such a
sequence $\psi$ exists, that $\psi > 0$, and that being convex, 
$\psi$ is decreasing, hence $\psi \in  c_o^*$ and by Lemma \ref {lem:3.6} ,
 $\omega\psi = \eta_{a_\infty}$ for some $\eta
\in \ell_1^*$. By definition, $\xi \geq \eta_{a_\infty}$ and hence
$(\eta _{a_\infty} ) \subset  (\xi ) = I$. To prove the reverse
inclusion, first notice that the graph of $\psi$ (viewed as the
polygonal curve through the points $\{(n,\psi_n) \mid n \in
\mathbb{N}\}$) must have infinitely many corners since $\psi_n
> 0$ for all $n$. Let $\{k_p\}$ be the strictly increasing sequence of all
the integers where the corners occur, starting with $k_1 = 1$, i.e.,
for all $p > 1$, $\psi_{k_p-1} - \psi_{k_p}
> \psi_{k_p}
- \psi_{k_p+1}$.
By the pointwise maximality of the convex sequence
$\psi$, $\psi_{k_p} = (\xo )_{k_p}$ for every $p \in \mathbb{N}$
(including $p =1$) since otherwise we could contradict maximality by
increasing $\psi_{k_p}$ and still maintain convexity and
majorization by $\xo$. Denote by $D_{\frac12}$ the operator
$(D_{\frac12} \zeta)_j = \zeta_{2j}$ for $\zeta \in c_o^*$. We claim
that for every $j$, $(D_{\frac12} \xo )_j < 2\psi_j$. Assume
otherwise that there is a $j \geq 1$ such that $(\xo )_{2j} \geq
2\psi_j$ and let $p$ be the integer for which $k_p \leq j <
k_{p+1}$. Then $k_p < 2j$ and also $2j < k_{p+1}$ because otherwise
we would have the contradiction $2\psi_j \leq (\xo )_{2j} \leq (\xo
)_{k_{p+1}} = \psi_{k_{p+1}} \leq \psi_j$.
Moreover, since $k_p$ and $k_{p+1}$ are consecutive corners, between
them $\psi$ is linear, i.e.,
\[
\psi_j = \psi_{k_p} +
\frac{\psi_{k_{p+1}}-\psi_{k_p}}{k_{p+1}-k_p}(j-k_p) =
\frac{k_{p+1}-j}{k_{p+1}-k_p}\psi_{k_p}+
\frac{j-k_p}{k_{p+1}-k_p}\psi_{k_{p+1}}
\]
and hence
\[
\psi_j \geq \frac{k_{p+1}-j}{k_{p+1}-k_p} \psi_{k_p}>
\bigg(1-\frac{j}{k_{p+1}}\bigg)\psi_{k_p} > \frac12 \bigg(\xo\bigg)_{\negmedspace k_p} \geq
\frac12 \bigg(\xo\bigg)_{\negmedspace 2j} \geq \psi_j.
\]

This contradiction proves that $D_{\frac12} \xo < 2\psi$. It is now
easy to verify that  for $j > 1$, $(\xo )_j \leq (D_3D_{\frac12} \xo
)_j < 2(D_3\psi)_j$ and hence $I = (\xi ) \subset (\eta_{a_\infty})$ because
\[
\xi_j < 2\omega_j(D_3\psi)_j \leq
2(D_3\omega)_j(D_3\psi)_j = 2(D_3(\omega\psi))_j =
2D_3(\eta_{a_\infty})_j.
\]
\end{proof}

\begin{example}\label{ex:3.8}
In the proof of the implication (iii) $\Rightarrow$ (ii), one cannot
conclude that $\xi  = O(\eta _{a_\infty} )$. Indeed consider $\xi_j
= \frac{1}{jk!}$ for $k! \leq j < (k+1)!$ where it is elementary to
compute $\psi_j = \frac{1}{k!} (1 - \frac{j-k!}{(k+1)!})$ for $k!
\leq j < (k+1)!$. Also, this example shows that while in the am-case
the smallest concave sequence $\frac{\eta_a}{\omega}$ that majorizes
$\xo$ (when $\xo$ is monotone nondecreasing) provides also the
smallest arithmetic mean $\eta_a$ that majorizes $\xi$  (see
Remark~\ref{rem:2.15}(iii)), this is no longer true for the
am-$\infty$ case.
\end{example}

We have seen in Lemma~\ref{lem:2.14} that the am-interior of a
nonzero principal ideal is always principal and it is nonzero if and
only if the ideal is large enough (that is, it contains $(\omega)$).
Furthermore, there is always a smallest am-open ideal containing it
and it too is principal. The next lemma shows that the am-$\infty$
interior of a nonzero principal ideal is principal if only if the
ideal is small enough (that is, it does not contain $(\omega)$).
Furthermore, if the principal ideal is contained in $\se(\omega)$,
which is the largest am-$\infty$  open ideal, then there is a
smallest am-$\infty$ open ideal containing it and it is principal.
\vspace{0.8cm}
\begin{lemma}\label{lem:3.9} For every $\xi \in  c_o^*$:
\item[(i)]
$(\xi)^{o\infty} = \begin{cases}(\omega\lni\xo) & \text{if }\omega\not\subset(\xi)\\
\se(\omega) & \text{if }\omega\subset(\xi)\end{cases}$

\item[(ii)] If $(\xi ) \subset  \se(\omega)$, then $(\omega \uni\xo )$ is the
smallest am-$\infty$ open ideal containing $(\xi )$.
\end{lemma}

\begin{proof}
(i) If $(\xi)=F$, then also $(\omega\lni\xo) = (\omega \uni\xo )=F$, so assume that $\xi\not\in \Sigma (F)$.
If $(\omega  ) \subset  (\xi )$, then $\se (\omega) = (\xi)^{o\infty}$ 
because $\se (\omega)$ is the largest am-$\infty$ open ideal.
If $(\omega)\not\subset(\xi)$, in particular $\omega \neq
O(\xi )$ and hence $\lni\xo \in  c_o^*$. But then by
Lemma~\ref{lem:3.7}, $(\omega \lni\xo ) = (\eta _{a_\infty} )$ for
some $\eta \in  \ell_1^*$, and since $(\eta _{a_\infty} ) =
(\eta)_{a_\infty}$ by \cite[Lemma~4.7]{10},
it follows that $(\omega
\lni\xo )$ is am-$\infty$  open. Since $\omega \lni\xo \leq \xi$ and
hence $(\omega \lni\xo ) \subset  (\xi )$, it follows that $(\omega
\lni\xo ) \subset (\xi )^{o\infty}$ . For the reverse inclusion, if
$\zeta \in \Sigma((\xi )^{o\infty})$, then $\zeta \leq
\rho_{a_\infty} \leq MD_m\xi$ for some $\rho \in \ell_1^*$, $M > 0$
and $m \in \mathbb{N}$. But then  $\frac{\rho_{a_\infty}}{\omega} \leq \lni
M\frac{D_m\xi}{\omega}$ because $\frac{\rho_{a_\infty}}{\omega}$ is
monotone nonincreasing,  and from this and $\omega
\leq D_m\omega \leq m\omega$,
it follows that
\begin{align*}
\rho_{a_\infty} &\leq M\omega
\lni\frac{D_m\xi}{\omega} \leq mM\omega \lni D_m \bigg(\xo\bigg) =
mM\omega D_m \lni\xo\\
&\leq mM (D_m\omega)\bigg(D_m\lni\xo \bigg) = mMD_m\bigg(\omega
\lni\xo \bigg)
\end{align*}
where the first equality follows by an elementary computation.
Thus $\zeta \in  \Sigma (\omega \lni \xo )$, i.e., $(\xi )^{o\infty}
\subset (\omega \lni\xo )$ and the equality of these ideals is
established.

(ii) If $(\xi ) \subset \se(\omega)$, then $\uni\xo \in c_o^*$,
hence $(\omega \uni\xo )$ is am-$\infty$  open by
Lemma~\ref{lem:3.7}. Clearly, $(\xi ) \subset  (\omega  \uni\xo   )$
and if $(\xi ) \subset I$ for an am-$\infty$  open ideal $I$, then
$\xi \leq \rho_{a_\infty}$  for some $\rho_{a_\infty} \in
\Sigma(I)$. Since $\frac{\rho_{a_\infty}}{\omega}$ is monotone
nonincreasing, by the minimality of ``uni'', $\omega\uni\xo\leq
\rho_{a_\infty}$  and hence $(\omega \uni\xo) \subset I$.
\end{proof}

As a consequence of this lemma we see that if $(\omega) = (\xi ) + (\eta)$ but $(\omega) \not\subset
(\xi )$ and $(\omega) \not\subset (\eta)$ as in Example \ref{ex:2.4}(ii), then $(\omega)^{o\infty} = \se(\omega)$ is
not principal but
\[
(\xi )^{o\infty}  + (\eta)^{o\infty} = \bigg(\omega \lni\xo \bigg) +
\bigg(\omega \lni\frac{\eta}{\omega}\bigg) = \bigg(\omega \lni\xo +
\omega \lni\frac{\eta}{\omega} \bigg)
\]
which is principal. By the same token, $\text{$_{a_\infty}$}(\xi ) +
\text{$_{a_\infty}$}(\eta) \neq \text{$_{a_\infty}$}((\xi ) +
(\eta))$ and in view of Theorem~\ref{thm:3.4}, $(\xi )_{-\infty}  +
(\eta)_{-\infty} \neq ((\xi ) + (\eta))_{-\infty}$.

>From this lemma we obtain an analog of Corollary~\ref{cor:2.16}.

\begin{corollary}\label{cor:3.10} Let
 $I$ be an ideal. Then
\item[(i)] \begin{align*}
\Sigma(I^{o\infty}) &= \{\xi  \in  \Sigma(\se (\omega))
\mid \omega \uni\xo \in \Sigma(I)\} \\
&= \{\xi  \in  c_o^* \mid \xi \leq \omega\lni\frac{\eta}{\omega} \text{ for some } 
\eta \in \Sigma(I \cap \se (\omega))\}.
\end{align*}

If $I$ is am-$\infty$  open and $\xi \in c_o^*$, then
\item[(ii)] $\xi  \in \Sigma(I)$ if and only if $\omega \uni\xo \in \Sigma(I)$.
\end{corollary}

\begin{proof} (i) If $\xi  \in \Sigma(I^{o\infty})$, then $\xi  \in  \Sigma(\se
(\omega))$ and hence $\omega\uni\xo \in
\Sigma(I^{o\infty}) \subset \Sigma(I)$ by Lemma~\ref{lem:3.9}(ii). If $\xi  \in  \Sigma(\se
(\omega))$ and $\omega \uni\xo   \in \Sigma(I)$, then $\omega \uni\xo
\in \Sigma(I \cap \se (\omega))$ and $\xi \leq \omega  \uni\xo =
\omega \lni \frac{\omega\uni\xo}{\omega}$. Thus 

\begin{align*}
\Sigma(I^{o\infty})
&\subset \{\xi \in \Sigma(\se (\omega)) \mid \omega\uni\xo \in \Sigma(I)\}\\
&\subset \{\xi  \in c_o^* \mid \xi \leq \omega
\lni\frac{\eta}{\omega} \text{ for some }\eta \in \Sigma(I \cap \se
(\omega))\}.
\end{align*}

 Finally, let $\xi  \in c_o^*$, $\xi \leq \omega
\lni\frac{\eta}{\omega}$ for some $\eta \in \Sigma(I \cap \se
(\omega))$. From the inequality $\xi  \leq \omega\uni\xo \leq \omega
\lni\frac{\eta}{\omega}$, it follows by  by Lemma~\ref{lem:3.9}(i) that  $\xi \in
\Sigma((\eta)^{o\infty} ) \subset \Sigma(I^{o\infty} )$, which
concludes the proof.

(ii) Just notice that $\xi \leq \omega\uni\xo   \in \Sigma(I)
\subset \Sigma(\se (\omega))$.
\end{proof}

Now Theorem~\ref{thm:2.17}, Definition~\ref{def:2.18} and
Proposition~\ref{prop:2.21} extend to the am-$\infty$  case with
proofs similar to the am-case.

\begin{theorem}\label{thm:3.11}
The intersection of am-$\infty$  open ideals is am-$\infty$  open.
\end{theorem}

\begin{deff}\label{def:3.12}
For every ideal $I$, define 
\[
I^{oo\infty} := \bigcap\{J \mid   I \cap \se (\omega) \subset J  ~\text{ and $J$
is  am-$\infty$ open}\}
\].
\end{deff}

\begin{remark}\label{rem:3.13}
Lemma~\ref{lem:3.9} affirms that if $I$ is principal then
$I^{o\infty}$  is principal if and only if
$(\omega)\not\subset(\xi)$ and $I^{oo\infty}$ is principal if and
only if $(\xi ) \subset \se(\omega)$.
\end{remark}

The next proposition generalizes to general ideals the
characterization of $I^{oo\infty}$ given by Lemma~\ref{lem:3.9} in
the case of principal ideals.

\begin{proposition}\label{prop:3.14}
For every ideal $I$, the characteristic set of $I^{oo\infty}$ is given by:
\[
\Sigma(I^{oo\infty}) = \bigg\{\xi  \in c_o^* \mid \eta \leq \omega
\uni\frac{\eta}{\omega} \text{ for some }\eta \in \Sigma(I \cap \se
(\omega))\bigg\}.
\]
\end{proposition}
\noindent Notice that $I^{o\infty} \subset I \cap \se (\omega) \subset I^{oo\infty}$ and
$I$ is am-$\infty$ open if and only if one of the inclusions and
hence both of them are equalities. Also, $I\cap \se (\omega) \subset
I_{a_\infty}$  and $I_{a_\infty}$  is am-$\infty$ open so
$I^{oo\infty} \subset I_{a_\infty}$. As for the am-case, we see by
considering an am-$\infty$  open principal ideal that is not
am-$\infty$  stable that the inclusion may be proper, and by
considering am-$\infty$  stable ideals that it may become an
equality.

\begin{example}\label{ex:3.15}
Let $\xi_j = \frac{1}{2^kk!}$  for $(k-1)! < j \leq k!$ for $k
> 1$. Then a direct computation shows that $(\uni\xo )_j = \frac{1}{2^k}$ for
$(k-1)! < j \leq k!$  and that $\uni\xo \asymp
\frac{\xi_{a_\infty}}{\omega}$. Thus by Lemma~\ref{lem:3.9}, $(\xi
)^{oo\infty}  = (\xi )_{a_\infty}$. On the other hand,
$(\frac{\omega\uni\xo}{\xi})_{(k-1)!} = k$ and hence $\xi_{a_\infty}
\neq O(\xi )$. By \cite[Theorem 4.12]{10}, $\xi$  is
$\infty$-irregular, i.e., $(\xi ) \neq (\xi )_{a_\infty}$.
\end{example}

A consequence of Proposition~\ref{prop:3.14} and the subadditivity
of ``uni''  is that for any two ideals $I$ and $J$, $I^{oo\infty}+ J^{oo\infty}  = (I +
J)^{oo\infty}$.

Proposition~\ref{prop:3.14} also permits us to determine simple
sufficient conditions on $I$ under which $I^{-\infty}$  (resp.,
$I^{oo\infty}$) is the largest am-$\infty$ closed ideal
$\mathcal{L}_1$ (resp.,
the largest am-$\infty$  open ideal $\se
(\omega)$).

\begin{lemma}\label{lem:3.16}
Let $I$ be an ideal.

\item[(i)] If $I \not\subset \mathcal{L}_1$, then $I^{-\infty}  =
\mathcal{L}_1$.

\item[(ii)] If $I \not\subset  \se (\omega)$, then $I^{oo\infty}  =
\se (\omega)$.

\end{lemma}

\begin{proof}
(i) Let $\xi  \in  \Sigma(I) \setminus \ell^*_1$. Then
$\se (\omega) = (\xi )_{a_\infty} \subset I_{a_\infty}$ by
\cite[Lemma 4.7]{10}. Since $I_{a_\infty} \subset \se (\omega)$
holds generally, $I_{a_\infty} = \se (\omega)$ and thus $I^{-\infty}
=\, _{a_\infty} \se (\omega) = \mathcal{L}_1$.

(ii) Let $\eta \in \Sigma(\se(\omega))$, set $\alpha :=
\uni\frac{\eta}{\omega}$, $\alpha_o = \alpha_1$, and 
choose an arbitrary $\xi \in  \Sigma(I) \setminus
\Sigma(\se (\omega))$. Then there is an increasing sequence of
integers $n_k$ with $n_0 = 0$ and an $\varepsilon > 0$ such that
$\xi_{n_k} \geq \varepsilon\omega_{n_k}$ for all $k \geq 1$. Set
$\mu_j = \frac{1}{n_k}$ and $\rho_j = \frac{\alpha_{n_{k-1}}}{n_k}$
for $n_{k-1} < j \leq n_k$ and $k \geq 1$. Then $\mu, \rho \in
c_o^*$, $\mu \leq \frac{1}{\varepsilon}\xi$, hence $\mu  \in
\Sigma(I)$ and $\rho = o(\omega)$, $\rho \leq \alpha_1\mu$, hence
$\rho \in \Sigma(I \cap \se (\omega))$. Moreover, $\max\{(
\frac{\rho}{\omega})_j \mid n_{k-1} < j \leq n_k\} =
\alpha_{n_{k-1}}$ and thus $(\uni \frac{\rho}{\omega})_j =
\alpha_{n_{k-1}}$ for $n_{k-1} < j \leq n_k$. But then, $\alpha \leq
\uni \frac{\rho}{\omega}$ and hence $\eta \leq \alpha\omega \leq
\omega \uni \frac{\rho}{\omega}$. By Proposition~\ref{prop:3.14},
$\eta \in \Sigma(I^{oo\infty})$, which proves the claim.
\end{proof}

Finally, it is easy to see that the exact analog of
Lemma~\ref{lem:2.22} holds.

\section{Soft Ideals}\label{sec:4}

It is well-known that the product $IJ=JI$ of two ideals $I$ and $J$
is the ideal with characteristic set
\[
\Sigma(IJ) = \{\xi \in c_o^* \mid \xi \leq \eta \rho \text{ for some
}\eta \in  \Sigma(I) \text{ and }\rho\in \Sigma(J)\}
\]
and that for all $p > 0$, the ideal $I^p$ is the ideal with
characteristic set 
\[
\Sigma(I^p) = \{\xi \in  c_o^* \mid \xi^{1/p} \in
\Sigma(I)\}
\]
 (see \cite[Section 2.8]{7} as but one convenient
reference). Recall also from \cite[Sections 2.8 and 4.3]{7} that the
quotient $\Sigma(I) : X$ of a characteristic set $\Sigma(I)$ by a
nonempty subset $X\subset [0,\infty)^{\mathbb N}$ is defined to be the
characteristic set
\[
\big\{\xi  \in  c_o^* \mid \big((D_m\xi)x\big)^* \in \Sigma(I)
\text{ for all }x \in  X \text{ and }m \in \mathbb{N}\big\}.
\]

Whenever $X = \Sigma(J)$, denote the associated ideal by $I : J$. 
A special important case is the K\"{o}the dual
$X^\times$ of a set $X$, which is the ideal with characteristic
set $\ell_1^* : X$.

In \cite{9} and \cite{10} we introduced the following definitions of
soft ideals.

\begin{deff}\label{def:4.1}
The soft interior of an ideal $I$ is the product $\se I: = IK(H)$,\linebreak
i.e., the ideal with characteristic set
\[
\Sigma(\se I) = \{\xi  \in  c_o^* \mid \xi  \leq \alpha \eta
\text{ for some } \alpha \in  c_o^*,\ \eta \in  \Sigma(I)\}.
\]
The soft cover of an ideal $I$ is the quotient $\scop I:= I : K(H)$,
i.e., the ideal with characteristic set
\[
\Sigma(\scop I) = \{\xi  \in  c_o^* \mid \alpha\xi  \in  \Sigma(I)
\text{ for all } \alpha \in  c_o^*\}.
\]
An ideal is called soft-edged if $\se I = I$ and soft-complemented
if $\scop I = I$.

\noindent A pair of ideals $I \subset J$ is called a soft pair if $\se J = I$
and $\scop I = J$.
\end{deff}

This terminology is motivated by the fact that $I$ is soft-edged if
and only if, for every $\xi  \in  \Sigma(I)$, one has $\xi  =
o(\eta)$ for some $\eta \in  \Sigma(I)$. Analogously, an ideal $I$
is soft-complemented if and only if, for every $\xi  \in  c_o^*
\setminus \Sigma(I)$, one has $\eta = o(\xi)$ for some $\eta \in
c_o^* \setminus \Sigma(I)$.

Below are some simple properties of the soft interior and soft cover
operations that we shall use frequently throughout this paper.

\begin{lemma}\label{lem:4.2}
For all ideals $I$, $J$:

\item[(i)]
$\se$ and $\scop$ are inclusion preserving, i.e., 
$\se I \subset \se J$ and $\scop I \subset  \scop J$ whenever  $I \subset J$.
\item[(ii)]
$\se$ and $\scop$ are idempotent, i.e., $\se (\se I) = \se I$ and
$\scop (\scop I) = \scop I$ and so $\se I$ and $\scop I$ are,
respectively, soft-edged and soft-complemented.
\item[(iii)]
$\se I \subset I \subset  \scop I$
\item[(iv)]
$\se (\scop I) = \se I$ and $\scop (\se I) = \scop I$
\item[(v)]
$\se I$ and $\scop I$ form a soft pair.
\item[(vi)]
If $I \subset J$ form a soft pair and $L$ is an intermediate ideal,
$I \subset  L \subset  J$, then $I = \se L$ and $J = \scop L$.
\item[(vii)]
If $I \subset J$, $I = \se I$, and $J = \scop J$, then $I$ and $J$
form a soft pair if and only if $\scop I = J$ if and only if $\se J = I$. 
\end{lemma}

\begin{proof}
(i) and (iii) follow easily from the definitions. From $K(H) =
K(H)^2$ follows the idempotence of $\se$ in the first part of (ii)
and the inclusion $\scop (\scop I) \subset  \scop I$, while the
equality here follows from (iii) and (i). That $\se (\scop I)
\subset I \subset  \scop (\se I)$ is immediate by
Definition~\ref{def:4.1}. Applying $\se$ to the first inclusion, by
(i)--(iii) follows the first equality in (iv) and the second
equality follows similarly. (v), (vi) and (vii) are now immediate.
\end{proof}

An easy consequence of this proposition and of
Definition~\ref{def:4.1} is:

\begin{corollary} \label{cor:4.3}
For every ideal $I$,

\item[(i)]
$\se I$ is the largest soft-edged ideal contained in $I$ and it is
the smallest ideal whose soft cover contains $I$

\item[(ii)]
$\scop I$ is the smallest soft-complemented ideal containing $I$ and
it is the largest ideal whose soft interior is contained in $I$.

\end{corollary}

The rest of this section is devoted to showing that many ideals in the
literature are soft-edged or soft-complemented (or both) and
that soft pairs occur naturally. Rather than proving directly
soft-complementedness, it is sometimes easier to prove a stronger
property:

\begin{deff}\label{def:4.4}
An ideal $I$ is said to be strongly soft-complemented (ssc for short) if
for every countable collection of sequences $\{\eta^{(k)}\} \subset
c_o^* \setminus \Sigma(I)$ there is a sequence of indices $n_k \in
\mathbb{N}$ such that $\xi  \not\in \Sigma(I)$ whenever $\xi  \in  c_o^*$ and
$\xi_i \geq \eta_i^{(k)}$ for all $k$ and for all $1 \leq i \leq n_k$.
\end{deff}

\begin{proposition}\label{prop:4.5}
Strongly soft-complemented ideals are soft-complemented.
\end{proposition}

\begin{proof}
Let $I$ be an ssc ideal, let $\eta \not\in \Sigma(I)$, and for each
$k \in \mathbb{N}$, set $\eta^{(k)}:= \frac{1}{k}\eta$.
Since
$\{\eta^{(k)}\} \subset  c_o^* \setminus \Sigma(I)$, there is an
associated sequence of indices $n_k$ which, without loss of
generality, can be taken to be strictly increasing. Set $n_o = 0$
and define $\alpha_i := \frac{1}{k}$ for $n_{k-1} < i \leq n_k$.
Then $\alpha \in  c_o^*$ and $(\alpha\eta)_i \geq \eta_i^{(k)}$ for
all $1 \leq i \leq n_k$ and all $k$. Therefore $\alpha\eta \not\in
\Sigma(I)$ and hence, by the remark following
Definition~\ref{def:4.1}, $I$ is soft-complemented.
\end{proof}

Example~\ref{ex:4.15}
and Proposition~\ref{prop:5.3} provide soft-complemented ideals that
are not strongly soft-complemented.

\begin{proposition}\label{prop:4.6} {\rm (i)} Countably generated
ideals are strongly soft-complemented and hence soft-complemented.

{\rm (ii)} If $I$ is a countably generated ideal and if
$\{\rho^{(k)}\}$ is a sequence of generators for its characteristic
set $\Sigma(I)$, then $I$ is soft-edged if and only if for every $k
\in \mathbb{N}$ there are $m, k' \in \mathbb{N}$ for which
$\rho^{(k)} = o(D_m\rho^{(k')})$. In particular, a principal ideal
$(\rho)$ is soft-edged if and only if $\rho = o(D_m\rho)$ for some
$m \in \mathbb{N}$. If a principal ideal
$(\rho)$ is soft-edged, then $(\rho) \subset \mathcal{L}_1$.
\end{proposition}

\begin{proof}
(i) As in the proof of Lemma~\ref{lem:2.8}, choose a 
sequence of generators $\rho^{(k)}$ for $\Sigma(I)$ with $\rho^{(k)}
\leq \rho^{(k+1)}$ and such that $\xi  \in  \Sigma(I)$ if and only
if $\xi  = O(\rho^{(m)})$ for some $m \in  \mathbb{N}$. Let
$\{\eta^{(k)}\} \subset  c_o^* \setminus \Sigma(I)$. Then, in
particular, $\eta^{(k)} \neq O(\rho^{(k)})$ for every $k$. Thus
there is a strictly increasing sequence of indices $n_k \in
\mathbb{N}$ such that $\eta_{n_k}^{(k)} \geq k\rho_{n_k}^{(k)}$ for
all $k$. If $\xi  \in  c_o^*$ and for each $k$, $\xi_i \geq
\eta_i^{(k)}$ for all $1 \leq i \leq n_k$, then for all $k \geq m$,
$\xi_{n_k} \geq \eta_{n_k}^{(k)} \geq k\rho_{n_k}^{(k)} \geq
k\rho_{n_k}^{(m)}$. Hence $\xi  \neq O(\rho^{(m)})$ for each $m$ and
thus $\xi  \not\in \Sigma(I)$, establishing that $I$ is ssc.

(ii) Assume that $I$ is soft-edged and let $k \in  \mathbb{N}$. 
By the remarks following Definition~\ref{def:4.1}, $\rho^{(k)} = o(\xi)$
for some $\xi \in \Sigma(I)$. But also $\xi  = O(D_m\rho^{(k')})$
for some $m$ and $k'$ and hence $\rho^{(k)} = o(D_m\rho^{(k')})$.
Conversely, assume that the condition holds and let $\xi \in
\Sigma(I)$. Then $\xi = O(D_m\rho^{(k)})$ for some $m$ and $k$ and
$\rho^{(k)} = o(D_p\rho^{(k')})$ for some $p$ and $k'$. Since $D_m
D_p = D_{mp}$, one has
\[
\lim_n \frac{\big(D_m\rho^{(k)}\big)_n}{\big(D_{mp}\rho^{(k')}\big)_n}=
\lim_n \bigg(D_m\bigg(\frac{\rho^{(k)}}{D_p\rho^{(k')}}\bigg)\bigg)_n=
\lim_j \bigg(\frac{\rho^{(k)}}{D_p\rho^{(k')}}\bigg)_j=
0,
\]
i.e., $D_m\rho^{(k)} = o(D_{mp}\rho^{(k')})$, whence $\xi  \in
\Sigma(\se I)$ and $I$ is soft-edged. Thus, if $I$ is a soft-edged
principal ideal with a generator $\rho$, then $\rho = o(D_m\rho)$
for some $m \in \mathbb{N}$. As a consequence, $\rho_{m^k} \leq
\frac{1}{m^2}\rho_{m^{k-1}}$ for $k$ large enough, from which it
follows that $\rho$ is summable.
\end{proof}

Next we consider \textbf{Banach ideals}. These are ideals that are complete
with respect to a symmetric norm (see for instance \cite[Section 4.5]{7})
and were called uniform-cross-norm ideals by Schatten \cite{18},
symmetrically normed ideals by Gohberg and Krein  \cite{8}, and
symmetric norm ideals by other authors. Recall that the norm of
$I$ induces on the finite rank ideal $F$ (or, more precisely, on $S(F)$,
the associated space of sequences of $c_o$ with finite support) a
symmetric norming function $\phi$, and the latter permits one to
construct the so-called minimal and maximal Banach ideals
$\mathfrak{S}_\phi^{(o)} = \cl(F)$ contained in $I$ (the closure taken
in the norm of $I$) and $\mathfrak{S}_\phi$ containing $I$ where
\begin{align*}
\Sigma\big(\mathfrak{S}_\phi\big) &= \big\{\xi  \in  c_o^* \mid \phi(\xi):=\sup
\phi\big( \langle \xi_1, \xi_2, \dots , \xi_n, 0, 0, \ldots \rangle \big) <
\infty \big\}\\
\Sigma\big(\mathfrak{S}_\phi^{(o)}\big) &= \big\{\xi  \in
\Sigma\big(\mathfrak{S}_\phi\big) \mid \phi\big( \langle \xi_n,
\xi_{n+1}, \ldots \rangle \big) \longrightarrow 0\big\}.
\end{align*}

As the following proposition implies, the ideals
$\mathfrak{S}_\phi^{(o)}$ and $\mathfrak{S}_\phi$ can be obtained
from $I$ through a ``soft'' operation, i.e.,
$\mathfrak{S}_\phi^{(o)} = \se I$ and $\mathfrak{S}_\phi = \scop I$,
and the embedding $\mathfrak{S}_\phi^{(o)} \subset
\mathfrak{S}_\phi$ is a natural example of a soft pair. In
particular, if $I$ is a Banach ideal, then so also are $\se I$ and $\scop I$.

\begin{proposition}\label{prop:4.7}
For every symmetric norming function $\phi$, $\mathfrak{S}_\phi^{(o)}$
is soft-edged,
$\mathfrak{S}_\phi$ is ssc, and $\mathfrak{S}_\phi^{(o)} \subset
\mathfrak{S}_\phi$ is a soft pair.
\end{proposition}

\begin{proof} We first prove that $\mathfrak{S}_\phi^{(o)}$ is
soft-edged. For every $\xi  \in \Sigma(\mathfrak{S}_\phi^{(o)})$, that is,\linebreak 
$\phi( \langle \xi_n, \xi_{n+1}, \ldots \rangle ) \rightarrow 0$, choose 
a strictly increasing sequence of indices $n_k$ with $n_o = 0$ for which
$\phi( \langle \xi_{n_k+1}, \xi_{n_k+2}, \ldots \rangle ) \leq
2^{-k}$ and $k\xi_{n_k}\downarrow 0$. Set  $\beta_i :=
k$ for all $n_{k-1} < i \leq n_k$ and $\eta  := \lni \beta\xi$. Then
$\eta \in c_o^*$ since $\eta_{n_k} \leq \beta_{n_k} \xi_{n_k} =
k\xi_{n_k} \rightarrow 0$ and $\xi  = o(\eta)$ because for every $k$
and $n_{k-1} < n \leq n_k$,
\begin{align*}
\eta_n&=\min \{\beta_i\xi_i \mid i \leq n\}\\
&=\min \!\Big\{\big\{\!\min \big\{j\xi_i \!\mid\! n_{j-1} < i \leq n_j\big\} \mid
1 \leq j \leq k-1\big\},\; \min \big\{k\xi_i \mid n_{k-1} < i \leq n\big\}\!\Big\}\\
&=\min
\big\{\big\{j\xi_{n_j} \mid 1 \leq j \leq k-1\big\},\ k\xi_n\big\}\\
&=\min \big\{(k-1)\xi_{n_{k-1}},\ k\xi_n\big\}\\
&\geq (k-1)\xi_n.
\end{align*}

Furthermore, $\eta \in  \Sigma(\mathfrak{S}_\phi^{(o)})$ which
establishes that $\mathfrak{S}_\phi^{(o)}$ is soft-edged. Indeed, for
all $h > k > 1$,
\begin{align*}
\phi\big( \langle \eta_{n_k+1}, \dots , \eta_{n_h}, 0, 0, \ldots \rangle \big)
&\leq \sum_{j=k}^{h-1}
\phi\big( \langle \eta_{n_j+1}, \dots , \eta_{n_{j+1}}, 0, 0, \ldots \rangle \big)\\
&\leq \sum_{j=k}^{h-1} (j+1)
\phi\big( \langle \xi_{n_j+1}, \dots , \xi_{n_{j+1}}, 0, 0, \ldots \rangle \big)\\
&\leq \sum_{j=k}^{h-1} (j+1)
\phi\big( \langle \xi_{n_j+1}, \xi_{n_j+2}, \ldots \rangle \big)\\
&\leq \sum_{j=k}^{h-1} \frac{j+1}{2^j}.
\end{align*}

Thus
\begin{align*}
\phi\big(\langle \eta_{n_k+1}, \eta_{n_k+2}, \ldots\rangle
\big)&= \sup_n\phi\big(\langle \eta_{n_k+1}, \dots,\eta_n, 0,
0, \ldots \rangle \big)\\
&\leq \sum_{j=k}^{\infty} \frac{j+1}{2^j}
\longrightarrow 0 \quad \text{as } k \longrightarrow \infty.
\end{align*}
from which it follows that $\phi( \langle \eta_n, \eta_{n+1}, \ldots
\rangle ) \rightarrow 0$ as $n \rightarrow \infty$.

Next we prove that $\mathfrak{S}_\phi$ is ssc. For every $\{\eta^{(k)}\}
\subset c_o^* \setminus \Sigma(\mathfrak{S}_\phi)$, that is,\linebreak $\sup_n
\phi( \langle \eta_1^{(k)},\eta_2^{(k)}, \dots ,\eta_n^{(k)}, 0, 0,
\ldots \rangle ) = \infty$ for each $k$, choose a strictly
increasing sequence of indices $n_k \in \mathbb{N}$ for which $\phi(
\langle \eta_1^{(k)}, \dots ,\eta_{n_k}^{(k)}, 0, 0, \ldots \rangle
) \geq k$. Thus, if $\xi  \in c_o^*$ and for each $k, \xi_i \geq
\eta_i^{(k)}$ for all $1 \leq i \leq n_k$, then $\phi( \langle
\xi_1, \xi_2, \dots , \xi_{n_k}, 0, 0, \ldots \rangle ) \geq k$ and
hence $\xi \not\in \Sigma(\mathfrak{S}_\phi)$, which shows that
$\mathfrak{S}_\phi$ is ssc.

Finally, to prove that $\mathfrak{S}_\phi^{(o)} \subset
\mathfrak{S}_\phi$ form a soft pair, in view of Lemma~\ref{lem:4.2}(vii),
Corollary~\ref{cor:4.3}(i) and the first two results in this proposition, it
suffices to show that $\se(\mathfrak{S}_\phi) \subset
\mathfrak{S}_\phi^{(o)}$. Let $\xi  \in  \Sigma(\se
(\mathfrak{S}_\phi))$, i.e., $\xi  \leq \alpha\eta$ for some $\alpha
\in c_o^*$ and $\eta  \in \Sigma(\mathfrak{S}_\phi)$. 
Then $\phi(\langle\xi_n, \xi_{n+1}, \ldots \rangle ) \leq \alpha_n\phi(\langle\eta_n, \eta_{n+1}, \ldots \rangle ) \leq \alpha_n\phi(\eta) \rightarrow 0$, 
i.e., $\xi  \in  \Sigma(\mathfrak{S}_\phi^{(o)})$.
\end{proof}

\begin{remark}\label{rem:4.8}
\item[(i)] In the notations of \cite{7} and of this paper, Gohberg and
Krein \cite{8} showed that the symmetric norming function
$\phi(\eta) := \sup \frac{\eta_a}{\xi_a}$ induces a complete norm on
the am-closure $(\xi )^-$ of the principal ideal $(\xi)$ and for this
norm
\[
\cl(F) = \mathfrak{S}_\phi^{(o)} \subset  \cl (\xi ) \subset
\mathfrak{S}_\phi = (\xi)^-.
\]
\item[(ii)] The fact that $\mathfrak{S}_\phi$ is soft-complemented was
obtained in \cite[Theorem 3.8]{17}, but Salinas proved only that
(in our notations) $\se \mathfrak{S}_\phi \subset \mathfrak{S}_\phi^{(o)}$
\cite[Remark 3.9]{17}. Varga reached the same conclusion in the case
of the am-closure of a principal ideal with a non-trace class
generator \cite[Remark 3]{19}.
\item[(iii)] By Lemma~\ref{lem:4.2}(vi), if $I$ is a Banach ideal such that
$\mathfrak{S}_\psi^{(o)} \subset I \subset \mathfrak{S}_\psi$ for
some symmetric norming function $\Psi$ and if $\phi$ is the symmetric
norming function induced by the norm of $I$ on $\Sigma(F)$, then
$\mathfrak{S}_\phi^{(o)} = \mathfrak{S}_\psi^{(o)}$ and
$\mathfrak{S}_\phi = \mathfrak{S}_\psi$ and hence $\phi$ and $\psi$ are
equivalent (cf. \cite[Chapter 3, Theorem 2.1]{8}).
\item[(iv)] The fact that $\mathfrak{S}_\phi^{(o)} \subset
\mathfrak{S}_\phi$ is always a soft pair yields immediately the
equivalence of parts (a)--(c) in \cite[Theorem 2.3]{17} without the need
to consider norms and hence establish (d) and (e). \vspace{12pt}

That $\mathfrak{S}_\phi^{(o)} \subset \mathfrak{S}_\phi$ is a soft
pair can help simplify the classical analysis of principal ideals.
In \cite[Theorem 3.23]{2} Allen and Shen used Salinas' results \cite{17} on
(second) K\"{o}the duals to prove that $(\xi ) = \cl (\xi )$ if and only if
$\xi$  is regular (i.e., $\xi \asymp \xi_a$, or in terms of ideals, if and
only if $(\xi)$ is am-stable). In \cite[Theorem 3]{19} Varga gave an
independent proof of the same result. This result is also a special
case of \cite[Theorem 2.36]{7}, obtained for countably generated ideals by
yet independent methods. A still different and perhaps simpler proof
of the same result follows immediately from Theorem~\ref{thm:2.11} and
the fact that $\mathfrak{S}_\phi^{(o)} \subset \mathfrak{S}_\phi$
form a soft pair.
\end{remark}

\begin{proposition}\label{prop:4.9}
$(\xi ) = \cl (\xi )$ if
and only if $\xi$  is regular.
\end{proposition}

\begin{proof}
The inclusion $\mathfrak{S}_\phi^{(o)} \subset  (\xi ) = \cl (\xi )
\subset \mathfrak{S}_\phi = (\xi )^-$ and the fact that $(\xi )$ is
soft complemented by Proposition~\ref{prop:4.6}(i), $\mathfrak{S}_\phi$ is
soft complemented by
Proposition~\ref{prop:4.7}, and $\mathfrak{S}_\phi^{(o)} \subset
\mathfrak{S}_\phi$ is a soft pair (ibid), proves by applying the
$\scop$ operation to the above inclusion that $(\xi ) = (\xi )^-$. The
conclusion now follows from Theorem~\ref{thm:2.11}.
\end{proof}

\begin{remark}\label{rem:4.10}
If $(\xi )^-$ is countably generated, so in particular if it is
principal, by Theorem~\ref{thm:2.11} it is am-stable and hence $(\xi
)^- = ((\xi )^-)_a = (\xi )_a = (\xi_a)$, so that $\xi_a$ is
regular. This implies that $\xi$  itself is regular, as was proven in
\cite[Theorem 3.10]{7} and as is implicit in \cite[Theorem IRR]{19}. 
This conclusion fails for general ideals: 
we construct in \cite {12} a non am-stable ideal with an am-closure that is countably
generated and hence am-stable by Theorem~\ref{thm:2.11}.
\end{remark}

Next we consider \textbf{Orlicz ideals} which provide another
natural example of soft pairs. Recall from \cite[Sections 2.37 and
4.7]{7} that if $M$ is a monotone nondecreasing function on $[0,
\infty)$ with $M(0) = 0$, then the small Orlicz ideal
$\mathcal{L}_M^{(o)}$ is the ideal with characteristic set $\{\xi
\in  c_o^* \mid \sum_n M(t\xi_n) < \infty \text{ for all } t > 0\}$ and
the Orlicz ideal $\mathcal{L}_M$ is the ideal with characteristic
set $\{\xi  \in  c_o^* \mid \sum_n M(t\xi_n) < \infty \text{ for some }
t > 0\}$.
If the function $M$ is convex, then $\mathcal{L}_M^{(o)}$
and $\mathcal{L}_M$ are respectively the ideals
$\mathfrak{S}_\phi^{(o)}$ and $\mathfrak{S}_\phi$ for the symmetric
norming function defined by
\[
\phi\big( \langle \xi_1, \xi_2, \dots , \xi_n, 0, 0, \ldots \rangle
\big) := \inf_{t >0} \bigg\{\frac{1}{t} \mid \sum_{i=1}^{n} M(t\xi_i) \leq 1\bigg\}.
\]
Thus, when $M$ is convex, $\mathcal{L}_M^{(o)} \subset \mathcal{L}_M$
form a soft pair by Proposition~\ref{prop:4.7}. In fact, the same
can be proven directly without assuming convexity for $M$.

\begin{proposition}\label{prop:4.11}
Let $M$ be a monotone nondecreasing function on $[0, \infty)$ with
$M(0) = 0$. Then $\mathcal{L}_M^{(o)}$ is soft-edged, $\mathcal{L}_M$ is
ssc, and $\mathcal{L}_M^{(o)} \subset \mathcal{L}_M$ is a soft
pair.
\end{proposition}

\begin{proof}
Take $\xi  \in  \Sigma(\mathcal{L}_M^{(o)})$ and choose a strictly
increasing sequence of indices $n_k \in  \mathbb{N}$ such that
$\sum_{i=n_{k-1}+1}^{\infty} M(k^2\xi_i) \leq 2^{-k}$ and
$k\xi_{n_k} \downarrow 0$. As in the proof of
Proposition~\ref{prop:4.7}, set $n_0 = 0$ and $\beta_i := k$ for all
$n_{k-1} < i \leq n_k$ and $\eta  := \lni \beta\xi$. Then $\eta \in
c_o^*$ and $\xi  = o(\eta)$. Let $t > 0$ be arbitrary and fix an
integer $k \geq t$. Then since $\eta \leq \beta\xi$ and $M$ is
monotone nondecreasing, it follows that
\begin{align*}
\sum_{i=n_k+1}^{\infty} M(t\eta_i) &\leq \sum_{i=n_k+1}^{\infty}
M(k\beta_i\xi_i) = \sum_{j=k+1}^{\infty} \sum_{i=n_{j-1}+1}^{n_j}
M(kj\xi_i)\\&\leq \sum_{j=k+1}^{\infty} \sum_{i=n_{j-1}+1}^{\infty}
M(j^2\xi_i) \leq \sum_{j=k+1}^{\infty} 2^{-j} < \infty.
\end{align*}
Therefore $\eta \in \Sigma(\mathcal{L}_M^{(o)})$, which proves that
$\mathcal{L}_M^{(o)}$ is soft-edged.

Next we prove that $\mathcal{L}_M$ is ssc. For every countable collection of
sequences $\eta^{(k)}\in  c_o^* \setminus \Sigma( \mathcal{L}_M)$,
since $\sum_i M(\frac{1}{k}\eta_i^{(k)}) = \infty$  for all $k$, we
can choose a strictly increasing sequence of indices $n_k \in
\mathbb{N}$ such that $\sum_{i=1}^{n_k} M(\frac{1}{k}\eta_i^{(k)})
\geq k$. If $\xi  \in  c_o^*$ and $\xi_i \geq \eta_i^{(k)}$ for all
$1 \leq i \leq n_k$, then for all $m$ and all $k \geq m$ it follows
that 
\[
\sum_{i=1}^{n_k} M(\frac{1}{m} \xi_i) \geq \sum_{i=1}^{n_k}
M(\frac{1}{k} \xi_i) \geq \sum_{i=1}^{n_k} M(\frac{1}{k}
\eta_i^{(k)}) \geq k
\]
 and hence $\sum_i M(t\xi_i) = \infty$  for all
$t > 0$. Thus $\xi  \not\in \Sigma(\mathcal{L}_M)$, which proves
that $\mathcal{L}_M$ is ssc.

To prove that $\mathcal{L}_M^{(o)} \subset  \mathcal{L}_M$ is a soft
pair, since $\mathcal{L}_M^{(o)}$ is soft-edged and $\mathcal{L}_M$
is soft-complemented, by Lemma~\ref{lem:4.2}(vii) it suffices to
prove that $\se \mathcal{L}_M \subset  \mathcal{L}_M^{(o)}$. Let
$\xi  \in \Sigma(\mathcal{L}_M)$, let $t_o > 0$ be such that $\sum_n
M(t_o\xi_n) < \infty$, and let $\alpha \in c_o^*$. For each $t > 0$
choose $N$ so that $t\alpha_n \leq t_o$ for $n \geq N$. By the
monotonicity of $M$, $\sum_{n=N}^{\infty} M(t\alpha_n\xi_n) <
\infty$ and hence $\alpha\xi \in \Sigma(\mathcal{L}_M^{(o)})$.
\end{proof}

The fact that $\mathcal{L}_M^{(o)} \subset  \mathcal{L}_M$ forms a
soft pair can simplify proofs of some properties of Orlicz ideals.
Indeed, together with \cite[Proposition 3.4]{10} that states that
for an ideal $I$, $\se I$ is am-stable if and only if $\scop I$ is
am-stable if and only if $I_a \subset  \scop I$, and combined with
Lemma~\ref{lem:4.16} below it yields an immediate proof of the
following results in \cite{7}: the equivalence of (a), (b), (c) in
Theorem~4.21 and hence the equivalence of (a), (b), (c) in
Theorem~6.25, the equivalence of (b), (c), and (d) in
Corollary~2.39, the equivalence of (b) and (c) in Corollary~2.40,
and the equivalence of (a), (b), and (c) in Theorem~3.21.

Next we consider \textbf{Lorentz ideals}. If $\phi$ is a monotone
nondecreasing nonnegative sequence satisfying the
$\Delta_2$-condition, i.e., $\sup\frac{\phi_{2n}}{\phi_n} <
\infty$, then in the notations of \cite[Sections 2.25 and 4.7]{7} the
Lorentz ideal $\mathcal{L}(\phi)$ corresponding to the sequence space
$\ell(\phi)$ is the ideal with characteristic set
\[
\Sigma(\mathcal{L}(\phi)) := \bigg\{\xi  \in  c_o^* \mid \|\xi
\|_{\ell(\phi)} := \sum_n \xi_n(\phi_{n+1} - \phi_n) <
\infty\bigg\}.
\]

A special case of Lorentz ideal is the trace class $\mathcal{L}_1$
which corresponds to the sequence $\phi = \langle n \rangle$ and the
sequence space $\ell(\phi) = \ell_1$.

Notice that $\mathcal{L}(\phi)$ is also the K\"{o}the dual
$\{\langle \phi_{n+1} - \phi_n \rangle\}^\times
= \ell_1^* : \{\langle  \phi_{n+1} - \phi_n \rangle\}$ of the
singleton set consisting of the sequence  $\langle\phi_{n+1} -
\phi_n \rangle$ (cf. \cite[Section 2.8(iv)]{7}). $\mathcal{L}(\phi)$
is a Banach ideal with norm induced by the cone norm $\| \cdot
\|_{\ell(\phi)}$ on $\ell(\phi)^*$ if and only if the sequence
$\phi$ is concave (cf. \cite[Lemma 2.29 and Section 4.7]{7}), and it
is easy to verify that in this case $\ell(\phi)^* =
\mathfrak{S}_{\psi}^{(o)} = \mathfrak{S}_{\psi}$ where $\psi$ is the
restriction of $\| \cdot \|_{\ell(\phi)}$ to $\Sigma(F)$. Thus by
Proposition~\ref{prop:4.7}, $\mathcal{L}(\phi)$ is both strongly
soft-complemented and soft-edged. In fact, the same holds without
the concavity assumption for $\phi$ as we see in the next
proposition.

\begin{proposition}\label{prop:4.12}
If $\phi$ be a monotone nondecreasing nonnegative sequence satisfying
the $\Delta_2$-condition, then $\mathcal{L}(\phi)$ is both soft-edged and strongly
soft-complemented.
\end{proposition}

\begin{proof} For $\xi  \in  \Sigma(\mathcal{L}(\phi))$, choose a strictly
increasing sequence of indices $n_k \in \mathbb{N}$ with
$k\xi_{n_k} \downarrow 0$ and $\sum_{i=n_k}^{\infty}
\xi_i(\phi_{i+1} - \phi_i) \leq 2^{-k}$. 
As in Proposition~\ref{prop:4.7}(proof), set $n_o = 0$, $\beta_i := k$ for all
$n_{k-1} < i \leq n_k$, hence  $\eta = \lni \beta\xi \in c_o^*$ and $\xi = o(\eta)$. Then
\[
\sum_{i}^{\infty}\eta_i(\phi_{i+1}-\phi_i) \!\leq\!
\sum_{i}^{\infty}\beta_i\xi_i(\phi_{i+1}-\phi_i)\!=\!
\sum_{k=1}^{\infty}\sum_{i=n_{k-1}+1}^{n_k}\!\!
k\xi_i(\phi_{i+1}-\phi_i) \!\leq\! \sum_{k=1}^{\infty}k2^{-k+1}
\!<\! \infty,
\]
whence $\eta \in \Sigma(\mathcal{L}(\phi))$. Thus $\xi  \in
\Sigma(\se \mathcal{L}(\phi))$ and hence $\mathcal{L}(\phi)$ is
soft-edged.

Finally, for every sequence of sequences $\{\eta^{(k)}\} \subset
c_o^* \setminus \Sigma(\mathcal{L}(\phi))$, choose a strictly
increasing sequence $n_k \in \mathbb{N}$ such that for all $k$, $
\sum_{i=1}^{n_k}\eta_i^{(k)}(\phi_{i+1} - \phi_i) \geq k$. Thus if
$\xi \in  c_o^*$ and $\xi_i \geq \eta_i^{(k)}$ for all $1 \leq i
\leq n_k$, then $\sum_{i=1}^{n_k} \xi_i(\phi_{i+1} - \phi_i) \geq k$
and hence $\xi \not\in \Sigma(\mathcal{L}(\phi))$, thus proving that
$\mathcal{L}(\phi)$ is ssc.
\end{proof}
\noindent In particular, we use frequently that $\mathcal{L}_1$ is both
soft-edged and soft-complemented.\vspace{12pt}

As the next proposition shows, any quotient with a soft-complemented
ideal as numerator is always soft-complemented (cf. first paragraph
of this section for the definition of quotient), but as
Example~\ref{ex:4.15} shows, even a K\"{o}the dual of a singleton
can fail to be strongly soft-complemented.

\begin{proposition}\label{prop:4.13}
Let $I$ be a soft-complemented ideal and let $X$ be a nonempty subset
of $[0,\!\infty)^{\mathbb N}$. Then the ideal with characteristic set
$\Sigma(I) : X$ is soft-complemented.
\end{proposition}

\begin{proof} Let $\xi  \in  c_o^* \setminus (\Sigma(I) : X)$, i.e.,
$((D_m\xi)x)^* \not\in \Sigma(I)$ for some $m \in \mathbb{N}$ and $x \in  X$.
As $I$ is soft-complemented, there exists $\alpha
\in  c_o^*$ such that $\alpha((D_m\xi )x)^* \not\in \Sigma(I)$. 
Let $\pi$ be an injection that monotonizes $(D_m\xi)x$, i.e.,
$(((D_m\xi)x)^*)_i = ((D_m\xi)x)_{\pi(i)}$ for all $i$. 
Define
\[
\gamma_j := \begin{cases}\alpha_{\pi^{-1}(j)} & \text{if
}j\in\pi(N)\\ 0 & \text{if }j\not\in \pi(N) \end{cases}. 
\]
Then
$\gamma \rightarrow 0$ and hence $\uni \gamma \in  c_o^*$. 
Thus for all $i$,
\begin{align*}
(\alpha((D_m\xi)x)^*)_i &= \gamma_{\pi(i)}(D_m\xi)_{\pi(i)}x_{\pi(i)}\\
&\leq (\uni \gamma)_{\pi(i)}(D_m\xi)_{\pi(i)}x_{\pi(i)}\\
&\leq
(D_m((\uni \gamma)\xi))_{\pi(i)}x_{\pi(i)}.
\end{align*}

\noindent From this inequality, and from the elementary fact that for two
sequence $\rho$ and $\mu$, $0 \leq \rho \leq \mu$  implies $\rho^*
\leq \mu^*$, it follows that $\alpha((D_m\xi)x)^* \leq ((D_m((\uni
\gamma)\xi))x)^*$. Thus $((D_m((\uni \gamma)\xi))x)^* \not\in \Sigma(I)$,
i.e., $(\uni \gamma)\xi  \not\in \Sigma(I) : X$, proving the claim.
\end{proof}

\begin{remark}\label{rem:4.14}
If $X$ is itself a characteristic set, the above result follows by
the simple identities for ideals $I$, $J$, $L$ analogous to the
numerical quotient operation ``$\div$'':
\[
(I : J) : L = I :(JL) = (I: L) : J
\]
 Indeed if in these identities we set $L = K(H)$ (the
ideal of compact operators), we obtain $\scop(I : J) = I : \se J =
\scop I : J$. Thus if $I$ is soft-complemented or $J$ is soft-edged it
follows that $I : J$ is soft-complemented. As an aside: 
\[
(I : J)J \subset I \subset (IJ : J) \subset I : J
\]
 and each of the embeddings can be proper (see also \cite{17}).
\end{remark}

\begin{example}\label{ex:4.15}
The K\"{o}the dual $I := \{\langle e^n \rangle\}^\times$ of the singleton 
$\{\langle e^n \rangle\}$  is soft-\\complemented by Proposition~\ref{prop:4.13} but
it is not strongly soft-complemented.  
Indeed, by definition, $\xi
\in \Sigma(I)$ if and only if $((D_m\xi)\langle e^n \rangle)^* \in
\ell_1^*$ (or, equivalently, $(D_m\xi)\langle e^n \rangle \in
\ell_1$) for every $m$, which in turns is equivalent to $\sum_n
\xi_n e^{mn} < \infty$ for every $m$. Choose $\eta \in  c_o^*$ such
that $\sum_n \eta_n e^n < \infty$ but $\sum_n \eta_n e^{2n} =
\infty$ and hence $\eta \not\in \Sigma(I)$, and set $\eta^{(k)} :=
D_{1/k}\eta$, i.e., $\eta_i^{(k)} = \eta_{ki}$ for all $i$. As
$(D_{2k}\eta^{(k)})_i \geq \eta_i$ for $i \geq k$, it follows that
for every $k$, $D_{2k}\eta^{(k)}$ and hence $\eta^{(k)}$ are not in
$\Sigma(I)$. 
Let $n_k \in \mathbb{N}$ be an arbitrary strictly
increasing sequence of indices, set $n_o = 0$ and define $\xi_i :=
\eta_i^{(k)}$ for $n_{k-1} < i \leq n_k$. As $\eta^{(k+1)} \leq
\eta^{(k)}$, it follows that $\xi$ is monotone nonincreasing and for
all $k$, $\xi_i \geq \eta_i^{(k)}$ for $1 \leq i \leq n_k$. On the
other hand, for all $m$ and for all $k \geq m$,
\[
\sum_{i=n_{k-1}+1}^{n_k} \xi_i e^{mi} \leq
\sum_{i=n_{k-1}+1}^{n_k} \eta_{ki} e^{ki} \leq
\sum_{i=kn_{k-1}+1}^{kn_k} \eta_i e^i
\]
and thus
\[
\sum_{i=n_{m-1}+1}^{\infty} \xi_i e^{mi} \leq
\sum_{i=n_{m-1}+1}^{\infty} \eta_i e^i < \infty,
\]
which proves that $\xi \in \Sigma(I)$ and hence that $I$ is not ssc.
\end{example}

Next we consider \textbf{idempotent ideals}, i.e., ideals for which
$I = I^2$. Notice that an ideal is idempotent if and only if $I = I^p$
for some $p \neq 0,1$, if and only if $I = I^p$ for all $p \neq 0$. The
following lemma is an immediate consequence of
Definition~\ref{def:4.1}, the remarks following it, and of
Definition~\ref{def:4.4}.

\begin{lemma}\label{lem:4.16} For every ideal $I$ and $p > 0$:
\item[(i)] $\se (I^p) = (\se I)^p$ and $\scop (I^p) = (\scop I)^p$

\noindent In particular, if $I$ is soft-edged or soft-complemented, then so respectively is $I^p$.

\item[(ii)] If $I \subset  J$ is a soft pair, then so is
$I^p \subset  J^p$.

\item[(iii)] If $I$ is ssc, then so is $I^p$.
\end{lemma}

\begin{proposition}\label{prop:4.17}

Idempotent ideals are both soft-edged
and soft-complemented.
\end{proposition}

\begin{proof}
Let $I$ be an idempotent ideal. That $I$ is soft-edged follows from
the inclusions $I = I^2\subset  K(H) I = \se I \subset  I$. That $I$
is soft-complemented follows from the inclusions
\[
\scop I = \scop (I^2) = (\scop I)^2 \subset  K(H) \scop I = \se
(\scop I) = \se I \subset I \subset  \scop I
\]
which follows from Lemmas \ref{lem:4.16} and \ref{lem:4.2}(iii),(iv).
\end{proof}

\noindent The remarks following Proposition~\ref{prop:5.3} show that idempotent ideals may fail to be strongly soft-complemented.
\vspace{12pt}

Finally, we consider the \textbf{Marcinkiewicz ideals} namely, the
pre-arithmetic means of principal ideals, and we consider also their
am-$\infty$  analogs. That these ideals are strongly
soft-complemented follows from the following proposition combined
with Proposition~\ref{prop:4.6}(i).

\begin{proposition} \label{prop:4.18}

The pre-arithmetic mean and the pre-arithmetic mean at infinity of a
strongly soft-complemented ideal is strongly soft-complemented. 

In particular, Marcinkiewicz ideals are strongly soft-complemented.
\end{proposition}

\begin{proof}

Let $I$ be an ssc ideal. We first prove that ${}_aI$ is ssc. Let
$\{\eta^{(k)}\} \subset  c_o^* \setminus \Sigma({}_aI)$, i.e.,
$\{\eta_a^{(k)}\} \subset c_o^* \setminus \Sigma(I)$, and let $n_k
\in \mathbb{N}$ be a strictly increasing sequence of indices for
which if $\zeta \in c_o^*$ and $\zeta_ i \geq {(\eta_a^{(k)})}_i$ for
all $1 \leq i \leq n_k$ and all $k$, then $\zeta  \not\in \Sigma(I)$.
Let $\xi \in c_o^*$ and $\xi_ i \geq {(\eta^{(k)})}_i$ for
all $1 \leq i \leq n_k$ and all $k$. But then ${(\xi_a)}_i \geq {(\eta_ a^{(k)})}_i$ for all
$1 \leq i \leq n_k$ and all $k$ and hence $\xi_a \not\in \Sigma(I)$,
i.e., $\xi \not\in \Sigma({}_aI)$.

We now prove that $_{a_\infty}I$ is ssc. Let $\{\eta^{(k)}\} \subset  c_o^*
\setminus \Sigma({}_{a_\infty} I)$. Assume first that infinitely
many of the sequences $\eta^{(k)}$ are not summable. Since the trace
class $\mathcal{L}_1$ is ssc by Proposition~\ref{prop:4.12}, there
is an associated increasing sequence of indices $n_k \in \mathbb{N}$
so that if $\xi \in c_o^*$ and $\xi_ i \geq \eta_i^{(k)}$ for all $1
\leq i \leq n_k$, then $\xi \not\in \Sigma(\mathcal{L}_1)$ and hence
$\xi \not\in \Sigma({}_{a_\infty} I)$ since ${}_{a_\infty} I \subset
\mathcal{L}_1$. Thus assume without loss of generality that all $\eta^{(k)}$ are summable and
hence $\eta_{a_\infty}^{(k)} \not\in \Sigma(I)$. Let $n_k \in
\mathbb{N}$ be a strictly increasing sequence of indices for which
$\zeta \not\in \Sigma(I)$ whenever $\zeta \in  c_o^*$ and $\zeta_i
\geq {(\eta_{ a_\infty}^{(k)})}_i$ for all $1 \leq i \leq n_k$ and
all $k$. For every $k$ and $n$ choose an integer $p(k,n) \geq n$ for
which $\sum^{p(k,n)}_{i=n} \eta_i^{(k)} \geq \frac 12
\sum^{\infty}_{i=n} \eta_i^{(k)}$. Set $N_k := \max\{p(k,n) \mid 1 \leq
n \leq n_k+1\}$. For any $\xi \in c_o^*$ such that $\xi _i \geq
\eta_i^{(k)}$ for all $1 \leq i \leq N_k$ consider two cases. If
$\xi$ is not summable then $\xi \not\in \Sigma({}_{a_\infty} I)$
trivially. If $\xi$ is summable, then for all $1 \leq n \leq n_k$
and for all $k$
\begin{align*}
{\big(\xi_{a_\infty} \big)}_n &=  \frac{1}{n} \sum_{i=n+1}^{\infty}
\xi_i \geq \frac 1n \sum_{i=n+1}^{N_k} \xi_i \geq \frac 1n
\sum_{i=n+1}^{N_k} \eta_i^{(k)} \\ &\geq \frac 1n
\sum_{i=n+1}^{p(k,n+1)} \eta_i^{(k)} \geq \frac 1{2n}
\sum_{i=n+1}^{\infty} \eta_i^{(k)}= \frac12
{\big(\eta_{a_\infty}^{(k)} \big)}_n
\end{align*}
and hence $\xi_{a_\infty} \not\in \Sigma(I)$, i.e., $\xi \not\in
\Sigma(_{a_\infty}I)$.
\end{proof}

That Marcinkiewicz ideals are ssc can be seen also by the following
consequence of Proposition~\ref{prop:4.7}. 
If $I$ is a Marcinkiewicz ideal, then $I = {}_a(\xi )=  {}_a((\xi )^o)$ for some $\xi \in c_o^*$. 
By Lemma \ref {lem:2.13},  $(\xi )^o = (\eta_a)= (\eta)_a$ for some $\eta\in c_o^*$. 
Thus  $I = {}_a((\eta )_a) = (\eta)^-$ and $(\eta)^-$ is ssc by Proposition~\ref{prop:4.7} and Remark~\ref{rem:4.8}(i).

Corollary~\ref{cor:6.7} and Proposition~\ref{prop:6.11} below show
that the pre-arithmetic mean (resp., the pre-arithmetic mean at
infinity) also preserve soft-complementedness. They also show that
the am-interior and the am-closure of a soft-edged ideal are
soft-edged, that the am-interior of a soft-complemented ideal is
soft-complemented by Proposition~\ref{prop:6.11}, and that the same holds for
the corresponding am-$\infty$ operations. However, as mentioned
prior to Proposition~\ref{prop:6.8}, (resp.,
Proposition~\ref{prop:6.11}) we do not know whether the am-closure
(resp., the am-$\infty$  closure) of a soft-complemented ideal is
soft-complemented. Likewise, we do not know whether the
am-closure (resp., am-$\infty$ closure) of an ssc ideal is ssc.

One non-trivial case in which we can prove directly that the
am-closure of an ssc ideal is scc is the following. If $I$ is countably
generated, then $I_a$ too is countably generated and hence, by
Propositions~\ref{prop:4.6}(i) and \ref{prop:4.18}(i), its
am-closure $I^-$ is also ssc, and then by Lemma~\ref{lem:4.16} so is
${(I^-)}^p$ for any $p > 0$. The latter ideal is in general not
countably generated (e.g., if $0 \neq \xi \in \ell_1^*$, then $( \xi
)^- = \mathcal{L}_1$ is not countably generated) but
Lemma~\ref{lem:4.19} below shows that nevertheless its am-closure is
ssc.

\begin{lemma}\label{lem:4.19}
For every ideal $I$, 
\[
{({(I^-)}^p)}^- = \begin{cases}{(I^-)}^p\quad&\text{for } 0 < p \leq
1\\  {(I^p)}^-&\text{for } p \geq 1.
\end{cases}
\]
\end{lemma}

\begin{proof}
Let $\xi  \in \Sigma (((I^-)^p)^-)$. By definition, $\xi_a \leq
\eta_a$ for some $\eta \in  \Sigma ((I^-)^p)$, i.e., $\eta^{1/p}\in
\Sigma(I^-)$, which in turns holds if and only if $(\eta^{1/p})_a
\leq \rho_a$ for some $\rho \in \Sigma(I)$.\linebreak Recall from
\cite[3.C.1.b]{16} that if $\mu$ and $\nu$ are monotone sequences
and $\mu_a \leq \nu_a$, then $(\mu^q)_a \leq (\nu^q)_a$ for $q \geq
1$. Thus, if $p \leq 1$, $(\xi^{1/p})_a \leq (\eta^{1/p})_a \leq
\rho_a$ and consequently $\xi^{1/p} \in \Sigma(I^-)$, i.e., $\xi \in
\Sigma((I^-)^p)$. Thus $((I^-)^p)^- \subset (I^-)^p$, which then
implies equality since the reverse inclusion is automatic. If $p
> 1$, the inequality $(\eta^{1/p})_a \leq \rho_a$\linebreak  implies for the same reason
that $\eta_a \leq (\rho^p)_a$. Hence $\xi_a \leq (\rho^p)_a$,
i.e.,  $\xi \in \Sigma((I^p)^-)$. Thus $((I^-)^p)^- \subset
(I^p)^-$, which then implies equality since the reverse inclusion is
again automatic.
\end{proof}

\begin{proposition}\label{prop:4.20}
If $I$ is countably generated and $0 <p < \infty$, then 
$((I^-)^p)^-$ is strongly soft-complemented.
\end{proposition}

\section{Operations on Soft Ideals}\label{sec:5}
In this section we investigate the soft interior and soft cover of
arbitrary intersections of ideals, unions of collections of ideals
directed by inclusion, and finite sums of ideals.

\begin{proposition}\label{prop:5.1}
For every collection of ideals $\{I_\gamma , \gamma \in \Gamma\}$:

\item[(i)]
 $\bigcap_\gamma  \se I_\gamma \supset \se (\bigcap_\gamma  I_\gamma)$

\item[(ii)]
 $\bigcap_\gamma  \scop I_\gamma = \scop (\bigcap_\gamma  I_\gamma)$ \\

In particular, the intersection of soft-complemented ideals is soft-complemented.
\end{proposition}

\begin{proof} (i) and the inclusion
$\bigcap_\gamma  \scop I_\gamma  \supset \scop (\bigcap_\gamma
I_\gamma)$ are immediate consequences of Lemma~\ref{lem:4.2}(i).
For the reverse inclusion in (ii), by (i) and Lemma~\ref{lem:4.2}
(i)--(iv) we have:
\[
\scop \left(\bigcap_\gamma  I_\gamma\right) \supset \bigcap_\gamma
I_\gamma \supset \bigcap_\gamma  \se I_\gamma  = \bigcap_\gamma  \se
(\scop I_\gamma ) \supset \se \left(\bigcap_\gamma \scop I_\gamma
\right)
\]
and hence
\[
\scop \left(\bigcap_\gamma  I_\gamma\right) \supset \scop
\left(\se\left(\bigcap_\gamma \scop I_\gamma\right)\right) = \scop
\left(\bigcap_\gamma \scop I_\gamma \right) \supset \bigcap_\gamma
\scop I_\gamma.
\]
\end{proof}

It follows directly from Definition~\ref{def:4.1} that if $\Gamma$
is finite, then equality holds in (i). In general, equality in (i)
does not hold, as seen in Example~\ref{ex:5.2} below, where the
intersection of soft-edged ideals fails to be soft-edged, thus
showing that the inclusion in (i) is proper.

\begin{example}\label{ex:5.2}
Let $\xi \in c_o^*$ be a sequence that satisfies the $\Delta_{1/2}$-condition, i.e.,
$\sup\frac{\xi_n}{\xi_{2n}} < \infty$, and let $\{I_\gamma\}_{ \gamma \in
\Gamma}$ be the collection of all soft-edged ideals containing the
principal ideal $(\xi)$. Then $I := \bigcap_\gamma  I_\gamma$ is not
soft-edged. Indeed, assume that it is and hence $\xi  = o(\eta)$ for
some $\eta \in \Sigma(I)$. By Lemma~\ref{lem:6.3} of the next
section, there is a sequence $\gamma  \uparrow\infty$  for which
$\gamma \leq \frac{\eta}{\xi}$ and $\mu:=\gamma \xi  \in  c_o^*$.
Then 
\[
(\xi ) \subset \se(\mu ) \subset  (\mu )
\subset (\eta) \subset I.
\]
Then $\se(\mu )\in \{I_\gamma\}_{ \gamma \in
\Gamma}$, hence $I \subset \se(\mu )$, and thus
$\se(\mu ) = (\mu )$. By Proposition~\ref{prop:4.6}(ii), this
implies that $\mu = o(D_m\mu)$ for some integer $m$. This is
impossible since \linebreak $\frac{\mu_n}{\mu_{2n}}=
\frac{\gamma_n}{\gamma_{2n}}\frac{\xi_n}{\xi_{2n}} \leq
\frac{\xi_n}{\xi_{2n}}$ which implies that $\mu$  too satisfies the
$\Delta_{1/2}$-condition and hence $D_m\mu = O(\mu )$, a
contradiction.
\end{example}

Notice that the conclusion that $\bigcap_\gamma I_\gamma$ is not
soft-edged follows likewise if $\{I_\gamma\}$ is a maximal chain of
soft-edged ideals that contain the principal ideal $(\xi)$.
Moreover, Example~\ref{ex:5.2} shows that in general there is no
smallest soft-edged cover of an ideal.

The next proposition shows that an intersection of strongly
soft-comple\-mented ideals, which is soft-complemented by
Proposition \ref{prop:5.1}(ii), can yet fail to be strongly
soft-complemented.

\begin{proposition}\label{prop:5.3}
The intersection of an infinite countable strictly decreasing chain
of principal ideals is never strongly soft-complemented.
\end{proposition}

\begin{proof} Let $\{I_k\}$ be the chain of principal ideals with $I_k \supsetneqq
I_{k+1}$ and set $I = \bigcap_k I_k$. First we find generators
$\eta^{(k)} \in c_o^*$ for the ideals $I_k$ such that$\eta^{(k)}
\geq \eta^{(k+1)}$. Assuming the construction up to
$\eta^{(k)}$, if $\xi$ is a generator of $I_{k+1}$ then $\xi \leq
MD_m\eta^{(k)}$ for some $M > 0$ and $m \in \mathbb{N}$. Set
$\eta^{(k+1)} := \frac{1}{M} D_{1/m}\xi$, where $(D_{1/m}\xi )_i =
\xi_{mi}$. Then $\eta^{(k+1)} \in c_o^*$ and $\eta^{(k+1)} \leq
\eta^{(k)}$ since $D_{1/m}D_m = id$. Moreover, $\eta^{(k+1)}\leq
\frac1M \xi$ and by an elementary computation, $\xi_i \leq
(D_{2m}D_{1/m}\xi )_i$ for $i \geq m$ so that $(\xi)\subset (\eta^{(k+1)})$ and hence 
$I_{k+1} = (\xi ) =(\eta^{(k+1)})$. By assumption, $\eta^{(k)} \not\in \Sigma(I)$ for
all $k$. For any given strictly increasing sequence of indices $n_k
\in \mathbb{N}$, set $n_o = 0$ and $\xi_i := \eta_i^{(k)}$ for
$n_{k-1} < i \leq n_k$. Since $\eta^{(k)} \geq \eta^{(k+1)}$ for all
$k$, it follows that $\xi \in c_o^*$ and $\xi_i \geq \eta_i^{(k)}$
for $1 \leq i \leq n_k$. Yet, since $\xi_i \leq \eta_i^{(k)}$ for
all $i \geq n_k$, one has $\xi \in \Sigma(\eta^{(k)})$ for all $k$
and hence $\xi \in \Sigma(I)$. Thus $I$ is not strongly
soft-complemented.
\end{proof}

Notice that if in the above construction $\eta^{(k)} = \rho^k$ for
some $\rho \in c_o^*$ that satisfies the $\Delta_{1/2}$-condition,
then $I = \bigcap_k(\rho^k)$ is also idempotent. This shows that
while idempotent ideals are soft-complemented by
Proposition~\ref{prop:4.17}, they can fail to be strongly
soft-complemented.

\begin{proposition}\label{prop:5.4}
For $\{I_\gamma\}_{ \gamma \in  \Gamma}$ a collection of ideals directed by inclusion:

\item[(i)]
$\bigcup_\gamma  \se I_\gamma  = \se ( \bigcup_\gamma I_\gamma)$\\

In particular, the directed union of soft edged ideals is soft-edged.

\item[(ii)]
$\bigcup_\gamma \scop I_\gamma  \subset  \scop ( \bigcup_\gamma I_\gamma)$

\end{proposition}

\begin{proof}
As in Proposition~\ref{prop:5.1}, (ii) and the inclusion
$\bigcup_\gamma \se I_\gamma  \subset  \se ( \bigcup_\gamma
I_\gamma)$ in (i) are immediate. For the reverse inclusion in (i),
from (ii) and Lemma~\ref{lem:4.2}(iii) and (iv) we have
\[
\se \left( \bigcup_ \gamma  I_\gamma\right) \subset  \se \left(
\bigcup_\gamma  \scop (\se I_\gamma)\right) \subset  \se \left(\scop
\left(\bigcup_\gamma  \se I_\gamma \right)\right) = \se \left(
\bigcup_\gamma  \se I_\gamma \right) \subset \bigcup_\gamma \se
I_\gamma.
\]
\end{proof}

It follows directly from Definition~\ref{def:4.1} that if $\Gamma$
is finite, then equality holds in (ii), but in general, it does not.
Indeed, any ideal $I$ is the union of the collection of all the
principal ideals contained in $I$ and this collection is directed by
inclusion since $(\eta) \subset I$ and $(\mu ) \subset I$ imply that
$(\eta), (\mu ) \subset  (\eta  + \mu ) \subset  I$. By
Proposition~\ref{prop:4.6}(i), principal ideals are ssc, hence
soft-complemented. Notice that by assuming the continuum hypothesis,
every ideal $I$ is the union of an increasing nest of countably
generated ideals \cite{3}, so then even nested unions of ssc ideals
can fail to be soft-complemented.

The smallest nonzero am-stable ideal $st^a(\mathcal{L}_1) =
\bigcup_{m=0}^{\infty}= (\omega)_{a^m}$ and the largest am-$\infty$
stable ideal $st_{a_\infty} (\mathcal{L}_1) = \bigcap_{m=0}^{\infty}
\text{$_{a_\infty^m}$}(\mathcal{L}_1)$ (see Section~\ref{sec:2})
play an important role in \cite{9,10}.

\begin{proposition}\label{prop:5.5}
The ideals $st^a(\mathcal{L}_1)$ and $st_{a_\infty} (\mathcal{L}_1)$
are both soft-edged and soft-complemented, $st^a(\mathcal{L}_1)$ is
ssc, but $st_{a_\infty}(\mathcal{L}_1)$ is not ssc.
\end{proposition}

\begin{proof}
For every natural number $m$, $(\omega)_{a^m} = (\omega_{a^m}) =
(\omega \log^m)$ is principal, hence $\Sigma(st^a(\mathcal{L}_1))$ is generated by the collection$ \{\omega\log^m\}_m.$ Since $\omega\log^m =
o(\omega \log^{m+1})$ for all $m$,  by
Proposition~\ref{prop:4.6}(i) and (ii), $st^a(\mathcal{L}_1)$ is
both soft-edged and ssc. From \cite[Proposition 4.17 (ii)]{10},
$st_{a_\infty} (\mathcal{L}_1) = \bigcap_{m=0}^{\infty}
\mathcal{L}(\sigma(\log^m))$, where using the notations of
\cite[Sections 2.1, 2.25, 4.7]{7}, $\mathcal{L}(\sigma(\log^m))$ is
the Lorentz ideal with characteristic set $\{\xi \in c_o^* \mid \xi
(\log)^m \in \ell_1\}$. Thus if $\xi \in
\Sigma(\bigcap_{m=0}^{\infty} \mathcal{L}(\sigma(\log^m)))$, then
also $\xi \log \in \Sigma(\bigcap_{m=0}^{\infty}
\mathcal{L}(\sigma(\log^m)))$ and hence $st_{a_\infty}
(\mathcal{L}_1)$ is soft-edged. By Propositions \ref{prop:4.12} and
\ref{prop:5.1}(ii), $st_{a_\infty} (\mathcal{L}_1)$ is
soft-complemented. However, $st_{a_\infty} (\mathcal{L}_1)$ is not
ssc. Indeed, set $\eta^{(k)} := \omega(\log)^{-k}$. Then $\eta^{(k)}
\not\in \Sigma(st_{a_\infty} (\mathcal{L}_1))$ for all $k$, but
$\eta^{(k)} \in \Sigma(\mathcal{L}(\sigma(\log^{k-2})))$ for each $k
\geq 2$. For any arbitrary sequence of increasing indices $n_k$, set
$n_o = 0$ and $\xi_j := (\eta^{(k)})_j$ for $n_{k-1} < j \leq n_k$.
Then $\xi_j \geq (\eta^{(k)})_j$ for $1 \leq j \leq n_k$ but also
$\xi_j \leq (\eta^{(k)})_j$ for $j \geq n_k$. Thus $\xi \in
\Sigma(\mathcal{L}(\sigma(\log^{k-2})))$ for all $k \geq 2$, hence
$\xi \in \Sigma(st_{a_\infty} (\mathcal{L}_1))$ which shows that
$st_{a_\infty} (\mathcal{L}_1)$ is not ssc.
\end{proof}

Now consider finite sums of ideals. Clearly, $K(H) (I + J) = K(H) I
+ K(H) J$, i.e., $\se (I + J) = \se I + \se J$ and hence finite sums
of soft-edged ideals are soft-edged.

The situation is far less simple for the soft-cover of a finite sum
of ideals. 
The inclusion $\scop (I + J) \supset \scop I + \scop J$
is trivial, but so far we are unable to determine whether or not
equality holds in general or, equivalently, whether or not the sum
of two soft-complemented ideals is always soft-complemented. We also
do not know if the sum of two ssc ideals is always
soft-complemented. However, the following lemma permits us to settle
the latter question in the affirmative when one of the ideals is
countably generated. Recall that if $0 \leq \lambda \in  c_o$, then
$\lambda^*$ denotes the decreasing rearrangement of $\lambda$.

\begin{lemma}\label{lem:5.6}
For all ideals $I$, $J$ and sequences $\xi \in c_o^*$: 

$\xi  \in \Sigma(I + J)$ \quad if and only if \quad $(\max((\xi - \rho), 0))^* \in \Sigma(I)$ for some $\rho \in \Sigma(J)$.
\end{lemma}

\begin{proof}
If $\xi  \in  \Sigma(I + J)$, then $\xi  \leq \zeta + \rho$ for some
$\zeta \in \Sigma(I)$ and $\rho \in  \Sigma(J)$. (Actually, one can
choose $\zeta$ and $\rho$ so that $\xi = \zeta + \rho$ but equality
is not needed here.) Thus $\xi - \rho \leq \zeta$, and so $\max((\xi
- \rho), 0) \leq \zeta$. But then, by the elementary fact that if
for two sequence $0 \leq \nu \leq \mu$, then $\nu^* \leq \mu^*$, it
follows that $\max((\xi - \rho), 0)^* \leq \zeta^* = \zeta$ and
hence $(\max((\xi - \rho), 0))^* \in  \Sigma(I)$. Conversely, assume
that $(\max((\xi  - \rho), 0))^* \in \Sigma(I)$ for some $\rho \in
\Sigma(J)$. Since $0 \leq \xi \leq \max((\xi - \rho), 0) + \rho$,
\[
\xi = \xi^* \leq (\max((\xi  - \rho), 0) + \rho)^* \leq D_2(\max((\xi
- \rho), 0)^*) + D_2\rho \in \Sigma(I + J),
\]
 where the second
inequality, follows from the fact that $(\rho + \mu )^* \leq
D_2\rho^* + D_2\mu^*$ for any two non-negative sequences $\rho$ and
$\mu$, which fact is likely to be previously known but is also the
commutative case of a theorem of K. Fan \cite[II Corollary 2.2,
Equation (2.12)]{8}.
\end{proof}

\begin{theorem}\label{thm:5.7}
The sum $I + J$ of an ssc ideal $I$ and a countably generated ideal
$J$ is ssc and hence soft-complemented.
\end{theorem}

\begin{proof}
As in the proof of Lemma~\ref{lem:2.8} there is an increasing
sequence of generators $\rho^{(k)} \leq \rho^{(k+1)}$ for the
characteristic set $\Sigma(J)$ such that $\mu  \in  \Sigma(J)$ if
and only if \linebreak $\mu  = O(\rho^{(m)})$ for some integer $m$. By passing
if necessary to the sequences $k\rho^{(k)}$, we can further assume
that $\mu  \in  \Sigma(J)$ if and only if $\mu  \leq \rho^{(m)}$ for
some integer $m$. Let $\{\eta^{(k)}\} \subset  c_o^* \setminus
\Sigma(I + J)$. By Lemma~\ref{lem:5.6}, for each $k$,
$(\max((\eta^{(k)} - \rho^{(k)}), 0))^* \not\in \Sigma(I)$ so, in
particular, $\eta_i^{(k)} > \rho_i^{(k)}$ for infinitely many
indices $i$. Let $\pi_k : \mathbb{N} \rightarrow \mathbb{N}$ be a
monotonizing injection for $\max((\eta^{(k)} - \rho^{(k)}), 0)$,
i.e., for all $i \in  \mathbb{N}$,
\[
\Big(\max\Big(\Big(\eta^{(k)} - \rho^{(k)}\Big), 0\Big)\Big)_i^* \!=\!
\Big(\max\Big(\Big(\eta^{(k)} - \rho^{(k)}\Big),
0\Big)\Big)_{\pi_k(i)} \!=\! \Big(\eta^{(k)} -
\rho^{(k)}\Big)_{\pi_k(i)} \!>\! 0.
\]

Since $I$ is ssc, there is a strictly increasing sequence of indices
$n_k \in \mathbb{N}$ such that if $\zeta \in c_o^*$ and $\zeta_i
\geq (\max((\eta^{(k)} - \rho^{(k)}), 0))_i^*$ for all $1\leq i \leq
n_k$, then $\zeta \not\in \Sigma(I)$. Choose integers $N_k \geq \max
\{\pi_k(i) \mid 1 \leq i \leq n_k\}$ so that $N_k$ is increasing. We
claim that if $\xi \in c_o^*$ and $\xi_i \geq \eta_i^{(k)}$ for all
$1 \leq i \leq N_k$ and all $k$, then $\xi  \not\in \Sigma(I + J)$,
which would conclude the proof. Indeed, for any given $m \in
\mathbb{N}$ and for each $k \geq m$, $1 \leq j \leq n_k$ and $1 \leq
i \leq j$, it follows that $\pi_k(i) \leq N_k$ and hence
\begin{align*}
\Big(\xi  - \rho^{(m)}\Big)_{\pi_k(i)} &\geq \Big(\eta^{(k)} -
\rho^{(k)}\Big)_{\pi_k(i)} = \Big(\max\Big(\Big(\eta^{(k)} -
\rho^{(k)}\Big), 0\Big)\Big)_i^*\\ &
\geq \Big(\max\Big(\Big(\eta^{(k)}
- \rho^{(k)}\Big), 0\Big)\Big)_j^*.
\end{align*}

Thus there are at least $j$ values of $(\xi - \rho^{(m)})_n$ that
are greater than or equal to $(\max((\eta^{(k)} - \rho^{(k)}),
0))_j^*$ and hence $(\max((\xi  - \rho^{(m)}), 0))_j^* \geq
(\max((\eta^{(k)} - \rho^{(k)}), 0))_j^*$. By the defining property
of the sequence $\{n_k\}$, $(\max((\xi  - \rho^{(m)}), 0))^* \not\in
\Sigma(I)$ for every $m$. But then, for any $\mu \in \Sigma(J)$
there is an $m$ such that $\mu  \leq \rho^{(m)}$ so that $(\max((\xi
- \mu ), 0))^* \geq (\max((\xi - \rho^{(m)}), 0))^*$ and hence
$(\max((\xi - \mu ), 0))^* \not\in \Sigma(I)$. 
By Lemma~\ref{lem:5.6}, it follows that $\xi \not\in \Sigma(I + J)$,
which concludes the proof of the claim and thus of the theorem.
\end{proof}

\section{Arithmetic Means and Soft Ideals}\label{sec:6}

The proofs of the main results in \cite[Theorems 7.1 and 7.2]{10}
depend in a crucial way on some of the commutation relations between
the se and sc operations and the pre and post-arithmetic means and
pre and post arithmetic means at infinity operations. In this
section we shall investigate these relations. We start with the
arithmetic mean and for completeness, we list the relations already
obtained in \cite[Lemma~3.3]{10} as parts (i)--(ii$'$) of the next
theorem.

\begin{theorem}\label{thm:6.1}
Let $I$ be an ideal.

\item[(i)] $\scop \text{$_a$}I \subset \text{$_a$}(\scop I)$

\item[(i$'$)] $\scop \text{$_a$}I = \text{$_a$}(\scop I) \text{ if and only if }
\omega \not\in \Sigma(\scop I) \setminus \Sigma(I)$

\item[(ii)] $\se I_a \subset (\se I)_a$

\item[(ii$'$)] $\se I_a = (\se I)_a\text{ if and only if }I = \{0\} \text{ or }I
\not\subset \mathcal{L}_1$

\item[(iii)] $\scop I_a \supset (\scop I)_a$

\item[(iv)] $\se \text{$_a$}I \supset \text{$_a$}(\se I)$

\item[(iv$'$)] $\se \text{$_a$}I = \text{$_a$}(\se I) \text{ if and only if }
\omega \not\in \Sigma(I) \setminus \Sigma(\se I)$.

\end{theorem}

The ``missing'' reverse inclusion of (iii) will be explored in
Proposition~\ref{prop:6.8}. The proof of parts (iii)--(iv$'$) of
Theorem~\ref{thm:6.1} depend on the following two lemmas.

\begin{lemma}\label{lem:6.2}\

\item[(i)] $F_a = (\mathcal{L}_1)_a = (\omega)$ and $\text{$_a$}(\omega) = \mathcal{L}_1$\\ 
Consequently $(\omega)$ and $\mathcal{L}_1$ are, respectively, the smallest nonzero am-open ideal and the smallest
nonzero am-closed ideal.

\item[(ii)] $\{0\} = \text{$_a$}I$ if and only if $\mathcal{L}_1 \not\subset
\text{$_a$}I$ if and only if $\omega \not\in \Sigma(I)$

\item[(iii)] $\mathcal{L}_1 = \text{$_a$}I$ if and only if $\omega  \in
\Sigma(I) \setminus \Sigma(\se I)$

\item[(iv)] $\mathcal{L}_1 \subsetneqq \text{$_a$}I$ if and only if $\omega \in
\Sigma(\se I)$

\end{lemma}

\begin{proof} Notice that $\eta_a \asymp \omega$ for every $0
\neq \eta \in \ell_1^*$  and that $\omega = o(\eta_a)$ for every
$\eta \not\in \ell_1^*$. Thus (ii) and the equalities in (i) follow
directly from the definitions. Recall from the paragraphs preceding
Lemma~\ref{lem:2.1} that an ideal is am-open (resp., am-closed) if
and only if it is the arithmetic mean of an ideal, in which case if
it is nonzero, it contains $F_a = (\omega)$ (resp., if and only if
it is the prearithmetic mean of an ideal, in which case by (ii), it
contains $\mathcal{L}_1$). Thus the minimality of $(\omega)$ (resp.,
$\mathcal{L}_1$) are established. (iii) follows immediately from
(ii) and (iv).

(iv) Assume first that $\mathcal{L}_1 \subsetneqq \text{$_a$}I$.
Then $\mathcal{L}_1 \subset  \se \text{$_a$}I$ since $\mathcal{L}_1$
is soft-edged (Proposition~\ref{prop:4.12}) and hence by (i),
\[
(\omega) = (\mathcal{L}_1)_a \subset  (\se\text{$_a$}I)_a = \se
((_aI)_a) = \se I^o \subset  \se I
\]
 where the second equality
follows from Theorem~\ref{thm:6.1}(ii$'$) applied to
$\text{$_a$}I$ which is not contained in $\mathcal{L}_1$. Conversely, assume that $\omega \in  \Sigma(\se I)$,
i.e., $\omega = o(\eta)$ for some $\eta \in  \Sigma(I)$. Then
$\mathcal{L}_1 \subset \text{$_a$}I$ by (ii). It follows directly
from the definition of $\lnd$ (see paragraph preceding
Lemma~\ref{lem:2.14}) that $\omega = o(\omega
\lnd\frac{\eta}{\omega})$. By Lemma~\ref{lem:2.14}(i),
$\omega\lnd\frac{\eta}{\omega}  \in  \Sigma(I^o)$, i.e.,
$\omega\lnd\frac{\eta}{\omega} \leq \rho_a \in  \Sigma(I)$ for some
$\rho \in \Sigma(_aI)$. But $\rho\not\in \ell_1^*$  since $\omega =
o(\rho_a)$ and hence $\mathcal{L}_1 \neq \text{$_a$}I$.
\end{proof}

\begin{lemma}\label{lem:6.3}
For $\eta \in  c_o^*$ and $0 < \beta \rightarrow \infty$, 
there is a sequence $0 < \gamma   \leq \beta$ with $\gamma \uparrow \infty$ for which $\gamma\eta$ is monotone nonincreasing.
\end{lemma}

\begin{proof}
The case where $\eta$ has finite support is elementary, so assume
that for all $i$, $\eta_i > 0$. By replacing if necessary $\beta$
with $\lnd \beta$ we can assume also that $\beta$ is monotone
nondecreasing. Starting with $\gamma_1 := \beta_1$, define
recursively 
\[
\gamma_n :=\frac{1}{\eta_n}\min(\gamma_{n-1}\eta_{n-1},\beta_n\eta_n).
\]
 It follows immediately that $\gamma \leq \beta$ and that $\gamma\eta$
is monotone nonincreasing. Moreover, $\gamma_n \geq \gamma_{n-1}$
for all $n$ since both $\beta_n \geq \beta_{n-1} \geq \gamma_{n-1}$
and $\gamma_{n-1}\frac{\eta_{n-1}}{\eta_n} \geq \gamma_{n-1}$. In
the case that $\gamma_n = \beta_n$ infinitely often, then
$\gamma\rightarrow \infty$. In the case that $\gamma_n \neq \beta_n$
for all $n > m$, then $\gamma_n\eta_n = \gamma_{n-1}\eta_{n-1}$ and
so also $\gamma_n = \frac{\eta_m}{\eta_n}\gamma_m \rightarrow
\infty$ since $\eta_n \rightarrow 0$ and $\eta_m\gamma_m \neq 0$.
\end{proof}

\begin{proof}[Proof of Theorem~\ref{thm:6.1}]

(i)--(ii$'$) See \cite[Lemma 3.3]{10}.

(iii) If $\xi  \in  \Sigma((\scop I)_a)$, then $\xi \leq \eta_a$ for
some $\eta \in \Sigma(\scop I)$. So for every $\alpha \in c_o^*$,
$\alpha\eta \in \Sigma(I)$ and $\alpha\xi \leq \alpha\eta_a \leq
(\alpha\eta)_a \in \Sigma(I_a)$, where the last inequality follows
from the monotonicity of $\alpha$. Thus $\xi \in \Sigma(\scop I_a)$.

(iv) Let $\xi \in \Sigma(_a(\se I))$, i.e., $\xi_a \leq \alpha\eta$
for some $\alpha \in c_o^*$ and $\eta \in \Sigma(I)$. Since
$(\frac{1}{\alpha}\xi)_a \leq \frac{1}{\alpha} \xi_a \leq \eta \in
c_o^*$ where the first inequality follows from the monotonicity of
$\alpha$, by Lemma~\ref{lem:6.3} there is a sequence $\gamma
\uparrow \infty$ such that $\gamma  \leq \frac{1}{\alpha}$ and
$\gamma \xi$ is monotone nonincreasing. Thus $(\gamma\xi )_a \leq
\eta \in \Sigma(I)$, i.e., $\gamma \xi \in \Sigma(_aI)$, and hence
$\xi \in \Sigma(\se \text{$_a$}I)$.

(iv$'$) There are three cases. If $\omega \not\in \Sigma(I)$, then
by Lemma~\ref{lem:6.2}(ii), both $\text{$_a$}I = \{0\}$ and
$\text{$_a$}(\se I) = \{0\}$ and hence the equality holds. If
$\omega \in \Sigma(I) \setminus \Sigma(\se I)$, then $\mathcal{L}_1
= \text{$_a$}I$ by Lemma~\ref{lem:6.2}(iii) and hence $\se
\text{$_a$}I = \mathcal{L}_1$ since $\mathcal{L}_1$ is soft-edged by
Proposition~\ref{prop:4.12}. But $\text{$_a$}(\se I) = \{0)$ by
Lemma~\ref{lem:6.2}(ii), so the inclusion in (iv) fails. For the
final case, if $\omega \in \Sigma(\se I)$, then by
Lemma~\ref{lem:6.2}(iv), $\mathcal{L}_1\subsetneqq \text{$_a$}I$.
Let $\xi \in \Sigma(\se \text{$_a$}I)$, i.e., $\xi = o(\eta)$ for
some $\eta \in \Sigma(_aI)$. By adding to $\eta$, if necessary, a
nonsummable sequence in $\Sigma(_aI)$, we can assume that $\eta$ is
itself not summable. But then it is easy to verify that $\xi_a =
o(\eta_a)$, i.e., $\xi_a \in \Sigma(\se I)$ and hence $\xi \in
\Sigma(_a(\se I))$.
\end{proof}

Now we examine how  the operations $\scop$
and $\se$ commute with the arithmetic mean operations of am-interior $I^o :=
(_aI)_a$ and am-closure $I^- := \text{$_a$}(I_a)$.

\begin{theorem}\label{thm:6.4}
Let $I$ be an ideal.

\item[(i)] $\scop I^- \supset (\scop I)^-$

\item[(ii)] $\se I^- = (\se I)^-$

\item[(iii)] $\scop I^o \subset (\scop I)^o$

\item[(iii$'$)] $\scop I^o = (\scop I)^o$ if and only if $\omega
\not\in \Sigma(\scop I) \setminus \Sigma(I)$

\item[(iv)] $\se I^o \supset (\se I)^o$

\item[(iv$'$)] $\se I^o = (\se I)^o$ if and only if $\omega \not\in \Sigma(I)
\setminus \Sigma(\se I)$

\end{theorem}

\begin{proof}
(i) The case $I = \{0\}$ is obvious. If $I \neq \{0\}$, then $\omega
\in \Sigma(I_a)$ and hence, by Theorem~\ref{thm:6.1}(i$'$) and (iii),
it follows that 
\[
\scop I^- = \scop \text{$_a$}(I_a) =
\text{$_a$}(\scop I_a) \supset \text{$_a$}((\scop I)_a) = (\scop
I)^-.
\]

(ii) There are three possible cases. The case when $I = \{0\}$ is again
obvious. In the second case when $\{0\} \neq I \subset \mathcal{L}_1$, then $I^- =
\mathcal{L}_1$ and $(\se I)^- = \mathcal{L}_1$ since $\mathcal{L}_1$
is the smallest nonzero am-closed ideal by Lemma~\ref{lem:6.2}(i).
Since $\mathcal{L}_1$ is soft-edged by Proposition~\ref{prop:4.12},
$\se I^- = \mathcal{L}_1$, so equality in (ii) holds. In the third case,
 $I \not\subset \mathcal{L}_1$. Then $\mathcal{L}_1
\subsetneqq I^-$ and $\omega \in \Sigma(\se I_a)$ by
Lemma~\ref{lem:6.2}(iv). Then 
\[
\se I^- = \se \text{$_a$}(I_a) =
\text{$_a$}(\se(I_a)) = \text{$_a$}((\se I)_a) = (\se I)^-
\]
 where
the second and third equalities follow from Theorem~\ref{thm:6.1}(iv$'$) and (ii$'$).

(iii) Let $\xi  \in \Sigma(\scop I^o)$ and let $\alpha \in c_o^*$.
By the definition of ``$\und$'' (see the paragraph preceding
Lemma~\ref{lem:2.14}) it follows easily that $\alpha\omega\und \xo
\leq \omega \und \frac{\alpha\xi}{\omega}$ and by
Corollary~\ref{cor:2.16}, that $\omega \und \frac{\alpha\xi}{\omega}
\in \Sigma(I)$ since $\alpha\xi  \in  \Sigma(I^o)$. Thus 
$\alpha\omega \und \xo \in \Sigma(I)$ and hence $\omega\und \xo \in
\Sigma(\scop I)$. But then, again by Corollary~\ref{cor:2.16},  $\xi \in \Sigma((\scop I)^o)$.

(iii$'$) If $\omega \not\in \Sigma(\scop I) \setminus \Sigma(I)$,
then 
\[
\scop I^o = \scop (_aI)_a \supset (\scop (_aI))_a = (_a(\scop
I))_a = (\se I)^o
\]
 by Theorem~\ref{thm:6.1}(iii) and (i$'$). If on
the other hand $\omega \in  \Sigma(\scop I) \setminus \Sigma(I)$,
then by Lemma~\ref{lem:6.2}(ii) $\text{$_a$}(\scop I) \neq \{0\}$
and hence $(\scop I)^o \neq \{0)$, while $\text{$_a$}(I) = \{0\}$
and hence $\scop (I)^o = \{0\}$.

(iv) and (iv$'$). There are three possible cases. If $\omega \not\in
\Sigma(I)$, then $I^o= \{0\}$ by Lemma~\ref{lem:6.2}(ii) and so
$\se I^o = \{0\}$ and $(\se I)^o = \{0\}$, i.e., (iv$'$) holds trivially. 
If $\omega \in \Sigma(I) \setminus \Sigma(\se I)$, then
$I^o \neq \{0\}$ and $(\se I)^o = \{0\}$ again by
Lemma~\ref{lem:6.2}(ii). But then $\se I^o \neq \{0\}$, so (iv)
holds but (iv$'$) does not. Finally, when $\omega \in \Sigma(\se I)$,
then $\mathcal{L}_1\subsetneqq \text{$_a$}I$ by Lemma~\ref{lem:6.2}
(iv) and hence 
\[
\se I^o = \se (_aI)_a = (\se (_aI))_a = (_a(\se
I))_a = (\se I)^o
\]
 by Theorem~\ref{thm:6.1}(ii$'$) and (iv$'$).
\end{proof}

We were unable to find natural conditions under which the reverse
inclusion of Theorem~\ref{thm:6.4}(i) holds (see also
Proposition~\ref{prop:6.8}), nor examples where it fails.

\begin{corollary}\label{cor:6.5} 
\item[(i)] If $I$ is an am-open ideal, then
$\scop I$ is am-open while $\se I$ is am-open if and only if $I\neq
(\omega)$.

\item[(ii)] If $I$ is an am-closed ideal, then $\scop I$ and $\se I$ are
am-closed.
\end{corollary}

\begin{proof} (ii) and the first implication in (i) are
immediate from Theorem~\ref{thm:6.4}.
For the second implication of
(i), assume that $I$ is am-open and that $0 \neq I \neq (\omega)$.
Then by Lemma~\ref{lem:6.2}(i), $(\omega) \subsetneqq I$ and
$\mathcal{L}_1  = \text{$_a$}(\omega) \subset \text{$_a$}I$. But
$\mathcal{L}_1 \neq \text{$_a$}I$ follows from  $(\mathcal{L}_1)_a =
(\omega) \neq I = (_aI)_a$. Then $\omega \in \Sigma(\se I)$ by
Lemma~\ref{lem:6.2}(iv), hence $\se I = \se I^o = (\se I)^o$ by
Theorem~\ref{thm:6.4}(iv$'$) and thus $\se I$ is am-open. If $I =
\{0\}$, then $\se I = \{0\}$ too is am-open. If $I = (\omega)$, then
$\se I \subsetneqq (\omega)$ cannot be am-open, again by
Lemma~\ref{lem:6.2}(i).
\end{proof}

For completeness' sake we list also some $\se$ and $\scop$
commutation properties for the largest am-closed ideal $I_-$
contained in $I$ and the smallest am-open ideal $I^{oo}$ containing
$I$ (see Corollary~\ref{cor:2.6} and Definition~\ref{def:2.18}).

\begin{proposition}\label{prop:6.6}
For every ideal $I$:

\item[(i)] $\scop I_- = (\scop I)_-$

\item[(ii)] $\se I_- \subset (\se I)_-$

\item[(iii)] $\scop I^{oo} \supset (\scop I)^{oo}$

\item[(iv)] $\se I^{oo} \subset (\se I)^{oo}$

\item[(iv$'$)] $\se I^{oo} = (\se I)^{oo}$ if and only if either $I = \{0\}$ or $I
\not\subset (\omega)$

\end{proposition}

\begin{proof}
(i)--(iii) Corollary~\ref{cor:6.5} and the maximality (resp.,
minimality) of $I_-$ (resp., $I^{oo}$) yield the inclusions $\scop
I_- \subset  (\scop I)_-$, $\se I_- \subset (\se I)_-$, and $\scop
I^{oo} \supset (\scop I)^{oo}$. From the second inclusion it follows
that 
\[
\se ((\scop I)_-) \subset (\se (\scop I))_- = (\se I)_-\subset
I_-
\]
 and hence 
\[
(\scop I)_- \subset  \scop(\scop I)_- = \scop (\se
((\scop I)_-)) \subset \scop I_-
\]
 so that equality holds in (i).

(iv) If $\eta \in \Sigma((\se I^{oo})$, then by
Proposition~\ref{prop:2.21}, $\eta \leq \alpha\omega \und \xo$ for
some $\xi \in \Sigma(I)$ and $\alpha \in  c_o^*$. As remarked in the
proof of Theorem~\ref{thm:6.4}(iii), it follows that $\eta\leq
\omega \und \frac{\alpha\xi}{\omega}$ and hence $\eta \in
\Sigma((\se I)^{oo})$, again by Proposition~\ref{prop:2.21}.

(iv$'$) There are three cases. If $I = \{0\}$, (iv$'$) holds
trivially. If $\{0\} \neq I \subset (\omega)$, then by the
minimality of $(\omega)$ among nonzero am-open ideals, $I^{oo} =
(\omega)$ and $(\se I)^{oo} = (\omega)$, so the inclusion in (iv$'$)
fails. If $I \not\subset (\omega)$, then $I^{oo} \neq (\omega)$ and
hence by Corollary~\ref{cor:6.5}(i), $\se I^{oo}$ is am-open and by
minimality of $(\se I)^{oo}$, (iv$'$) holds.
\end{proof}

It is now an easy application of the above results to verify that
the following am-operations preserve softness.

\begin{corollary}\label{cor:6.7}

\item[(i)] If $I$ is soft-complemented, then so are $\text{$_a$}I$, $I^o$, and
$I_-$.

\item[(ii)] If $I$ is soft-edged, then so are $\text{$_a$}I$, $I^o$, and $I^-$.

\item[(iii)] If $I$ is soft-edged, then $I_a$ is soft-edged if and only if
either $I = \{0\}$ or $I \not\subset \mathcal{L}_1$.

\item[(iv)] If $I$ is soft-edged, then $I^{oo}$ is soft-edged if and only
if either $I = \{0\}$ or $I \not\subset (\omega)$.

\end{corollary}

Several of the ``missing'' statements that remain open are equivalent as shown in the next proposition.

\begin{proposition}\label{prop:6.8}
For every ideal $I$, the following conditions are equivalent.

\item[(i)] $\scop I_a \subset (\scop I)_a$

\item[(ii)] $(\scop I)_a$ is soft-complemented

\item[(iii)] $(\scop I)^-$ is soft-complemented

\item[(iv)] $\scop I^- \subset (\scop I)^-$

\end{proposition}

\begin{proof}
Implications (i) $\Rightarrow$ (ii) $\Rightarrow$ (iii) $\Rightarrow$ (iv) are easy consequences of Theorem \ref{thm:6.1} and Corollary \ref{cor:6.7}. 
We prove that (iv) $\Rightarrow$ (i). The case $I = \{0\}$ being trivial, assume $I \neq \{0\}$. 
Then $\omega \in \Sigma(I_a)$, hence $\scop I^- \supset \text{$_a$}\scop (I_a)$ by Theorem~\ref{thm:6.1}(i$'$). 
Moreover, since $I_a$ is am-open, then so is $\scop I_a$ by Corollary~\ref{cor:6.5}, i.e., $\scop I_a = (\scop I_a)^o$. 
Then
\[
\scop I_a = (_a(\scop I_a))_a \subset  (\scop I^-)_a \subset ((\scop
I)^-)_a = (\scop I)_a,
\]
the latter equality following from the general identity $(_a(J_a))_a = J_a$.
\end{proof}

Now we investigate the relations between arithmetic means at
infinity and the $\se$ and $\scop$ operations and we list some
results already obtained in \cite[Lemma 4.19]{10} as parts (i) and
(ii) of the next theorem.

\begin{theorem}\label{thm:6.9}
For every ideal $I \neq \{0\}$:

\item[(i)] $\scop \text{$_{a_\infty}$}I= \text{$_{a_\infty}$}(\scop I)$

\item[(ii)] $\se I_{a_\infty} =(\se I)_{a_\infty}$

\item[(iii)] $\scop I_{a_\infty} \supset (\scop I)_{a_\infty}$

\item[(iv)] $\se \text{$_{a_\infty}$}I = \text{$_{a_\infty}$}(\se I)$

\end{theorem}

\begin{proof} (i)--(ii) See \cite[Lemma 4.19]{10}.

(iii) If $\xi  \in \Sigma((\scop I)_{a_\infty} )$, $\xi  \leq \eta_{a_\infty}$ for some $\eta \in  \Sigma(\scop I \cap \mathcal{L}_1)$. 
In \cite[Lemma 4.19 (i)]{10}(proof) we showed
that for every $\alpha \in c_o^*$, $\alpha\eta_{a_\infty} \leq
(\alpha'\eta)_{a_\infty}$ for some $\alpha' \in  c_o^*$. But then
$\alpha'\eta \in \Sigma(I \cap \mathcal{L}_1)$ and so $\alpha\xi
\leq (\alpha'\eta)_{a_\infty} \in \Sigma(I_{a_\infty})$, i.e., $\xi
\in \Sigma(\scop I_{a_\infty})$.

(iv) Let $\xi \in \Sigma (\se \text{$_{a_\infty}$}I)$, then $\xi
\leq \alpha\eta$ for some $\alpha \in c_o^*$ and $\eta \in
\Sigma(_{a_\infty} I)$. But then by the monotonicity of $\alpha$,
$\xi_{a_\infty} \leq (\alpha\eta)_{a_\infty}\leq
\alpha\eta_{a_\infty} \in \Sigma(\se I)$. Thus $\xi \in
\Sigma(_{a_\infty} (\se I))$ which proves the inclusion $\se
\text{$_{a_\infty}$}I \subset \text{$_{a_\infty}$}(\se I)$.

Now let $\xi \in \Sigma(_{a_\infty}  (\se I))$, i.e.,
$\xi_{a_\infty} \leq \alpha\eta$ for some $\alpha \in c_o^*$ and
$\eta \in \Sigma(I)$. We construct a sequence $\gamma \uparrow
\infty$ such that $\gamma  \xi$  is monotone nonincreasing and
$(\gamma\xi )_{a_\infty} \leq \eta$. Without loss generality assume
that $\xi_n \neq 0$ and hence $\alpha_n \neq 0$ for all $n$. We
choose a strictly increasing sequence of indices $n_k$ (with $n_o =
0$) such that for $k \geq 1$, $\alpha_{n_k} \leq 2^{-k-2}$ and
$\sum_{n_{k+1}+1}^{\infty} \xi_i \leq \frac14 \sum_{n_k+1}^{\infty}
\xi_i$ for all $k$. Set $\beta_n = 2^k$ for $n_k < n \leq n_{k+1}$.
Then for all $k \geq 0$ and $n_k < n+1 \leq n_{k+1}$ we have
\begin{align*}
\sum_{n+1}^{\infty} \beta_i\xi_i & = 2^k \sum_{n+1}^{n_{k+1}} \xi_i
+ 2^{k+1} \sum_{n_{k+1}+1}^{n_{k+2}} \xi_i + 2^{k+2}
\sum_{n_{k+2}+1}^{n_{k+3}} \xi_i + \cdots\\
& \leq 2^{k} \sum_{n+1}^{n_{k+1}} \xi_i + 2^{k+1} \left(
\sum_{n_{k+1}+1}^{\infty} \xi_i + 2 \sum_{n_{k+2}+1}^{\infty} \xi_i
+ 2^2 \sum_{n_{k+3}+1}^{\infty} \xi_i + \cdots \right)\\
&\leq 2^{k} \sum_{n+1}^{n_{k+1}} \xi_i + 2^{k+2}
\sum_{n_{k+1}+1}^{\infty} \xi_i \leq 2^{k+2} \sum_{n+1}^{\infty} \xi_i  \\
& \leq \frac{1}{\alpha_{n_k}} \sum_{n+1}^{\infty} \xi_i \leq
\frac{1}{\alpha_n}\sum_{n+1}^{\infty} \xi_i = \frac{n}{\alpha_n}
(\xi_{a_{\infty}})_n \leq n\eta_n.
\end{align*}
This proves that $(\beta\xi)_{a_\infty} \leq \eta$. Now Lemma~\ref{lem:6.3} provides a sequence $\gamma \leq \beta$,
with $\gamma \uparrow \infty$ and $\gamma  \xi$ monotone
nonincreasing, and hence $(\gamma\xi )_{a_\infty} \leq
(\beta\xi)_{a_\infty} \leq \eta$. Thus $\gamma \xi  \in
\Sigma(_{a_\infty} I)$ and hence $\xi  = \frac{1}{\gamma}
(\gamma\xi) \in \Sigma(\se \text{$_{a_\infty}$}I)$.
\end{proof}

The reverse inclusion in Theorem~\ref{thm:6.9}(iii) does not hold
in general. Indeed, whenever $I_{a_\infty}  = \se (\omega)$ (which
condition by \cite[Corollary 4.9 (ii)]{10} is equivalent to\linebreak 
$I^{-\infty} = \text{$_{a_\infty}$}(I_{a_\infty}) = \mathcal{L}_1$
and in particular is satisfied by $I = \mathcal{L}_1$), it follows
that $\scop I_{a_\infty} = (\omega)$ while $(\scop I)_{a_\infty}
\subset \se (\omega)$. We do not know of any natural sufficient
condition for the reverse inclusion in Theorem~\ref{thm:6.9}(iii)
to hold.

Many of the other results obtained for the arithmetic mean case have
an analog for the am-$\infty$ case:

\begin{theorem}\label{thm:6.10}
For every ideal $I$:

\item[(i)] $\scop I^{-\infty} \supset (\scop I)^{-\infty}$

\item[(ii)] $\se I^{-\infty} =(\se I)^{-\infty}$

\item[(iii)] $\scop I^{o\infty} \supset (\scop I)^{o\infty}$

\item[(iii$'$)] $\scop I^{o\infty} =(\scop I)^{o\infty}$  if and only if
$\scop I^{o\infty} \subset  \se (\omega)$

\item[(iv)] $\se I^{o\infty} =(\se I)^{o\infty}$

\end{theorem}

\begin{proof} (i), (ii), (iii), and (iv) follow immediately from
Theorem~\ref{thm:6.9}.

(iii$'$) Since every am-$\infty$ open ideal is contained in $\se
(\omega)$, it follows that\linebreak $\scop I^{o\infty} = (\scop I)^{o\infty}
\subset  \se (\omega)$. Assume now that $\scop I^{o\infty} \subset
\se (\omega)$, let $\xi  \in \Sigma(\scop I^{o\infty})$, and let
$\alpha \in c_o^*$. Since $\xi  = o(\omega)$, there is an increasing
sequence of integers $n_k$ with $n_o = 0$ for which $(\uni\xo)_j =
(\xo)_{n_k}$ for $n_{k-1} < j \leq n_k$. Define $\tilde{\alpha}_j =
\alpha_1$ for $1 < j \leq n_1$ and $\tilde{\alpha}_j = \alpha_{n_k}$
for $n_k < j \leq n_{k+1}$ for $k \geq 1$. Then $\tilde{\alpha}\in
c_o^*$ and for all $k \geq 1$ and $n_{k-1} < j \leq n_k$
\[
\bigg(\!\alpha \uni\xo\!\bigg)_j \!=\! \alpha_j\bigg(\!\xo\!\bigg)_{n_k} \!\leq\!
\alpha_{n_{k-1}}\bigg(\!\xo\! \bigg)_{n_k} \!=\!
\bigg(\frac{\tilde{\alpha}\xi}{\omega}\bigg)_{n_k} \!\leq\! \bigg(\uni
\frac{\tilde{\alpha}\xi}{\omega}\bigg)_{n_k} \!\leq\!
\bigg(\uni\frac{\tilde{\alpha}\xi}{\omega}\bigg)_j.
\]
Since $\tilde{\alpha}\xi \in \Sigma(I^{o\infty})$ by hypothesis, it
follows that $\omega\uni\frac{\tilde{\alpha}\xi}{\omega} \in
\Sigma(I)$ by Corollary~\ref{cor:3.10}. But then $\alpha\omega
\uni\xo \in \Sigma(I)$ for all $\alpha \in  c_o^*$, i.e., $\omega
\uni\xo \in \Sigma(\scop I)$. Hence, again by
Corollary~\ref{cor:3.10}, $\xi \in \Sigma((\scop I)^{o\infty})$ and
hence $\scop I^{o\infty}  \subset  (\scop I)^{o\infty}$. By (iii) we
have equality.
\end{proof}

The necessary and sufficient condition in Theorem~\ref{thm:6.10}
(iii$'$) is satisfied in the case of most interest, namely when $I
\subset \mathcal{L}_1$. As in the am-case, we know of no natural
conditions under which the reverse inclusion of (i) holds nor examples where it fails (see also
Proposition~\ref{prop:6.8}). In the
following proposition we collect the am-$\infty$  analogs of
Corollary~\ref{cor:6.5}, Proposition~\ref{prop:6.6}, and
Corollary~\ref{cor:6.7}. Recall by Lemma~\ref{lem:3.16} that
$I^{oo\infty} = \se (\omega)$ for any ideal $I \not\subset \se
(\omega)$.

\begin{proposition}\label{prop:6.11}
Let $I\neq \{0\}$ be an ideal.

\item[(i)] If $I$ is am-$\infty$  open, then so is $\se I$.

\item[(i$'$)] If $I$ is am-$\infty$ open, then $\scop I$ is am-open if and
only if $\scop I \subset \se (\omega)$.

\item[(ii)] If $I$ is am-$\infty$ closed, then so are $\se I$ and $\scop I$.

\item[(iii)] $\se I^{oo\infty}  = (\se I)^{oo\infty}$

\item[(iv)] $\scop I^{oo\infty} \supset (\scop I)^{oo\infty}$

\item[(v)] $\se I_{-\infty} \subset (\se I)_{-\infty}$

\item[(vi)] $\scop I_{-\infty} = (\scop I)_{-\infty}$

\item[(vii)] If $I$ is soft-edged, then so are $\text{$_{a_\infty}$}I$,
$I_{a_\infty}$, $I^{-\infty}$, $I^{o\infty}$, and $I^{oo\infty}$.

\item[(viii)] If $I$ is soft-complemented, then so is $\text{$_{a_\infty}$}I$
and $I_{-\infty}$.

\item[(viii$'$)] If $I$ is soft-complemented, then $I^{o\infty}$ is soft-complemented if and only if \[\scop I^{o\infty} \subset  \se(\omega).\]

\end{proposition}

\begin{proof} (i) Immediate from Theorem~\ref{thm:6.10}(iv).

(i$'$) If $\scop I \subset  \se (\omega)$ then $\scop I = (\scop
I)^{o\infty}$  by Theorem~\ref{thm:6.10}(iii$'$) and hence $\scop
I$ is am-$\infty$ open. The necessity is clear since $\se (\omega)$
is the largest am-$\infty$  open ideal.

(ii) $\se I$ is am-$\infty$ closed by Theorem~\ref{thm:6.10}(ii). By Theorem~\ref{thm:6.10}(i) 
 and the am-$\infty$  analog of the 5-chain of inclusions given in Section~\ref{sec:2}, 
\[
\scop I = \scop I^{-\infty} \supset
(\scop I)^{-\infty} \supset \scop I \cap \mathcal{L}_1 = \scop I,
\] 
where the last equality holds because 
$\mathcal{L}_1$ is the largest am-$\infty$ closed ideal so contains $I$, 
and being soft-complemented it contains $\scop I$.

(iii) By (i), $\se I^{oo\infty}$ is am-$\infty$ open and by
Definition~\ref{def:3.12} and  Proposition \ref {prop:5.1} and following remark,  it contains $\se (I \cap \se (\omega)) =
\se I \cap \se (\omega)$, hence it must contain $(\se
I)^{oo\infty}$. On the other hand, if $\xi \in \Sigma(\se
I^{oo\infty})$, then by Proposition~\ref{prop:3.14} there is an
$\alpha \in c_o^*$ and $\eta \in \Sigma(I\cap \se (\omega))$ such
that $\xi \leq \alpha\omega \uni\frac{\eta}{\omega}$. Then, by the
proof in Theorem~\ref{thm:6.10}(iii$'$), there is an
$\tilde{\alpha} \in c_o^*$ such that $\alpha\omega
\uni\frac{\eta}{\omega} \leq
\omega\uni\frac{\tilde{\alpha}\eta}{\omega}$. Since $\tilde{\alpha}
\eta \in \Sigma(\se I \cap \se (\omega))$,
Proposition~\ref{prop:3.14} yields again $\xi \in \Sigma((\se
I)^{oo\infty})$ which proves the equality in (iii).

(iv) Let $\xi \in \Sigma((\scop I)^{oo\infty})$. By
Proposition~\ref{prop:3.14} there is an $\eta \in  \Sigma((\scop I)
\cap \se (\omega))$ such that $\xi  \leq\omega
\uni\frac{\eta}{\omega}$. Then, by the proof in
Theorem~\ref{thm:6.10}(iii$'$), for every $\alpha \in c_o^*$ there
is an $\tilde{\alpha} \in  c_o^*$ such that $\alpha\xi \leq
\alpha\omega \uni\frac{\eta}{\omega}\leq \omega\uni
\frac{\tilde{\alpha}\eta}{\omega}$. As $\tilde{\alpha} \eta \in
\Sigma(I \cap \se (\omega))$, again by Proposition~\ref{prop:3.14},
$\alpha\xi \in \Sigma(I^{oo\infty})$ and hence $\xi  \in
\Sigma(\scop (I^{oo\infty}))$.

(v) This is an immediate consequence of (ii).

(vi) The inclusion $\scop I_{-\infty} \subset (\scop I)_{-\infty}$
is similarly an immediate consequences of (ii).
The reverse inclusion follows from (v) applied to the ideal $\scop I$:
\[
\se(\scop I)_{-\infty} \subset  (\se \scop I)_{-\infty} = (\se I)_{-\infty} 
\subset  I_{-\infty}
\]
hence 
\[
(\scop I)_{-\infty}\subset \scop (\scop I)_{-\infty} = \scop ( \se(\scop I)_{-\infty}) \subset  \scop I_{-\infty}.
\]

(vii) The first two statements follow from Theorem~\ref{thm:6.9}
(iv) and (ii), the next two from Theorem~\ref{thm:6.10}(ii) and
(iv), and the last one from (iii).

(viii), (viii$'$) follow respectively from Theorem~\ref{thm:6.9}(i)
and Theorem~\ref{thm:6.10}(iii$'$).

\end{proof}

\subsection*{Acknowledgments} We wish to thank Ken Davidson for his
input on the initial phase of the research and Daniel Beltita for
valuable suggestions on this paper.


\begin{thebibliography}{99}
\bibitem{1}
S. Albeverio, D. Guido, A. Posonov, and S. Scarlatti,
\textit{Singular traces and compact operators.} J. Funct. Anal.
\textbf{137} (1996), 281--302.

\bibitem{2}
G. D. Allen and L. C. Shen, \textit{On the structure of principal
ideals of operators.} Trans. Amer. Math. Soc. \textbf{238} (1978),
253--270.

\bibitem{3}
A. Blass and G. Weiss, \textit{A characterization and sum
decomposition for operator ideals.} Trans. Amer. Math. Soc.
\textbf{246} (1978), 407--417.

\bibitem{4}
C. Bennett and R. Sharpley, \textit{Interpolation of Operators},
Pure and Applied Mathematics, vol. 129, Academic Press, 1988.

\bibitem{5}
J. W. Calkin, \textit{Two-sided ideals and congruences in the ring
of bounded operators in Hilbert space.} Ann. of Math. (2)
\textbf{42} (1941), 839--873.

\bibitem{6}
K. Dykema, G. Weiss, and M. Wodzicki, \textit{Unitarily invariant
trace extensions beyond the trace class.} Complex analysis and
related topics (Cuernavaca, 1996), Oper. Theory Adv. Appl.
\textbf{114} (2000), 59--65.

\bibitem{7}
K. Dykema, T. Figiel, G. Weiss, and M. Wodzicki, \textit{The
commutator structure of operator ideals.} Adv. Math. \textbf{185/1}
(2004), 1--79.

\bibitem{8}
I. C. Gohberg and M. G. Krein, \textit{Introduction to the Theory of
Linear Nonselfadjoint Operators}, American Mathematical Society,
1969.

\bibitem{9}
V. Kaftal and G. Weiss, \textit{Traces, ideals, and
arithmetic means.} Proc. Natl. Acad. Sci. USA \textbf{99} (2002),
7356--7360.

\bibitem{10}
\bysame, \textit{Traces on operator ideals and arithmetic means},
preprint.

\bibitem{11}
\bysame, \textit{Majorization for infinite sequences, an extension of the Schur-Horn Theorem, and operator ideals,} in preparation.

\bibitem{12}
\bysame, \textit{B(H) Lattices, density, and arithmetic
mean ideals}, preprint.

\bibitem{12A}
\bysame, \textit {Second order arithmetic means  in operator ideals}, J. Operators and Matrices, to appear.

\bibitem{12B}
\emph{A survey on the interplay between arithmetic mean ideals, traces, lattices of operator ideals, and an infinite Schur-Horn majorization theorem,} Proceedings of the 21st International Conference on Operator Theory, Timisoara 2006 (Theta Bucharest), to appear

\bibitem{13}
N. J. Kalton, \textit{Unusual traces on operator ideals.} Math.
Nachr. \textbf{136} (1987), 119--130.

\bibitem{14}
\bysame, \textit{Trace-class operators and commutators.} J. Funct.
Anal. \textbf{86} (1989), 41--74.

\bibitem{15}
A. S. Markus, \textit{The eigen- and singular values of the sum and
product of linear operators.} Uspekhi Mat. Nauk \textbf{4} (1964),
93--123.

\bibitem{16}
A. Marshall and I. Olkin, \textit{Inequalities: Theory of
Majorization and its Applications, Mathematics in Science and
Engineering}, vol. 143, Academic Press, New York, 1979.

\bibitem{17}
N. Salinas, \textit{Symmetric norm ideals and relative conjugate
ideals.} Trans. Amer. Math. Soc. \textbf{138} (1974), 213--240.

\bibitem{18}
R. Schatten, \textit{Norm ideals of completely continuous
operators}, Ergebnisse der Mathematik und  irher Grenzgebiete, Neue
Folge, Heft 27, Springer Verlag, Berlin, 1960.

\bibitem{19}
J. Varga, \textit{Traces on irregular ideals.} Proc. Amer. Math.
Soc. \textbf{107} (1989), 715--723.

\bibitem{20}
G. Weiss, \textit{Commutators and Operator ideals}, dissertation
(1975), University of Michigan microfilm.

\bibitem{21}
M. Wodzicki, \textit{Vestigia investiganda.} Mosc. Math. J.
\textbf{4} (2002), 769--798, 806.

\end{thebibliography}
\end{document}